\pdfoutput=1
\RequirePackage{ifpdf}
\ifpdf % We~are running pdfTeX in pdf mode
\documentclass[pdftex]{sigma}%,draft
\else
\documentclass{sigma}
\fi

\numberwithin{equation}{section}

\newtheorem{Theorem}{Theorem}[section]
\newtheorem{Lemma}[Theorem]{Lemma}
 { \theoremstyle{definition}
\newtheorem{Example}[Theorem]{Example}
\newtheorem{Remark}[Theorem]{Remark} }

\def\a{\mathbf{a}}
\def\b{\mathbf{b}}
\def\cb{\mathbf{c}}
\def\d{\mathbf{d}}
\def\x{\mathbf{x}}
\def\y{\mathbf{y}}
\def\e{\mathbf{e}}
\def\f{\mathbf{f}}
\def\l{\mathbf{l}}
\def\C{\mathbb{C}}
\def\L{\mathcal{L}}
\def\f{\mathbf{f}}
\def\m{\mathbf{m}}
\def\n{\mathbf{n}}
\def\N{\mathbb{N}}

\def\Z{\mathbb{Z}}
\def\Be{\mathcal{B}}

\begin{document}
\allowdisplaybreaks

\newcommand{\arXivNumber}{2105.05196}

\renewcommand{\PaperNumber}{098}

\FirstPageHeading

\ShortArticleName{Hypergeometric Functions at Unit Argument}

\ArticleName{Hypergeometric Functions at Unit Argument:\\ Simple Derivation of Old and New Identities}

\Author{Asena \c{C}ET\.INKAYA~$^{\rm a}$, Dmitrii KARP~$^{\rm bc}$ and Elena PRILEPKINA~$^{\rm cd}$}

\AuthorNameForHeading{A.~\c{C}etinkaya, D.~Karp and E.~Prilepkina}

\Address{$^{\rm a)}$~\.Istanbul Kultur University, \.Istanbul, Turkey}
\EmailD{\href{mailto:asnfigen@hotmail.com}{asnfigen@hotmail.com}}

\Address{$^{\rm b)}$~Holon Institute of Technology, Holon, Israel}
\Address{$^{\rm c)}$~Far Eastern Federal University, Ajax Bay~10, Vladivostok, 690922, Russia}
\EmailD{\href{mailto:dimkrp@gmail.com}{dimkrp@gmail.com}}

\Address{$^{\rm d)}$~Institute of Applied Mathematics, FEBRAS, 7 Radio Street, Vladivostok, 690041, Russia}
\EmailD{\href{mailto:pril-elena@yandex.ru}{pril-elena@yandex.ru}}

\ArticleDates{Received May 20, 2021, in final form October 31, 2021; Published online November 07, 2021}

\vspace{-1mm}

\Abstract{The main goal of this paper is to derive a number of identities for the generalized hypergeometric function evaluated at unity and for certain terminating multivariate hypergeometric functions from the symmetries and other properties of Meijer's $G$ function. For instance, we recover two- and three-term Thomae relations for ${}_3F_2$, give two- and three-term transformations for ${}_4F_3$ with one unit shift and ${}_5F_4$ with two unit shifts in the parameters, establish multi-term identities for general ${}_{p}F_{p-1}$ and several transformations for terminating Kamp\'e de F\'eriet and Srivastava $F^{(3)}$ functions. We further present a presumably new formula for analytic continuation of ${}_pF_{p-1}(1)$ in parameters and reveal somewhat unexpected connections between the generalized hypergeometric functions and the generalized and ordinary Bernoulli polynomials. Finally, we exploit some recent duality relations for the generalized hypergeometric and $q$-hypergeometric functions to derive multi-term relations for terminating series.}

\vspace{-1mm}

\Keywords{generalized hypergeometric function; Meijer's $G$ function; multiple hyper\-geo\-met\-ric series; Kamp\'e de F\'eriet function; Srivastava function; hypergeometric identity; generalized Bernoulli polynomials}

\Classification{33C20; 33C60; 33C70}

\vspace{-3mm}

\section{Introduction and preliminaries}

Here and throughout the paper we will use the standard symbol ${_{p}F_q}(\a;\b;z)$ for the generalized hypergeometric function with parameter vectors $\a\in\C^{p}$, $\b\in\C^{q}\!\setminus\!\{0,-1,\dots\}$, see \cite[Section~2.1]{AAR}, \cite[Section~5.1]{LukeBook}, \cite[Sections 16.2--16.12]{NIST} for precise definitions and details. We will omit the argument $z=1$ from the above notation throughout the paper. The guiding idea of this work is to employ the properties of Meijer's $G$ function $G^{m,n}_{p,q}$ to discover new identities for the generalized hypergeometric function ${_{r+1}F_r}$. This idea proved very fruitful and appears as a recurrent theme in a series of papers published by the second and the third named authors over past decade, including \cite{KLJAT2017,KPSIGMA,KPITSF2017,KPGFmethod,KPChapter2019,LPKITSF2020}.

Let us introduce some notation and definitions. For any vectors $\a\in\C^{n}$, $\b\in\C^{m}$ and a~scalar~$\beta$ define
\begin{gather}
\Gamma(\a)=\prod_{i=1}^{n}\Gamma(a_i),\qquad \a+\beta=(a_1+\beta,\dots,a_n+\beta),\nonumber\\[-1ex]
\a_{[k]}=(a_1,\dots,a_{k-1},a_{k+1},\dots,a_n),
\qquad
 \nu(\a;\b)=\sum_{j=1}^{m}b_j-\sum_{i=1}^{n}a_{i},\label{eq:notation}
\end{gather}
where $\Gamma(\cdot)$ is Euler's gamma function. Given integers $0\leq{n}\leq{p}$, $0\leq{m}\leq{q}$ and complex vectors $\a\in\C^n$, $\b\in\C^{m}$, $\cb\in\C^{p-n}$, $\d\in\C^{q-m}$, such that $a_i-b_j\notin\N$ for all $i$, $j$, Meijer's $G$ function is defined by the Mellin--Barnes integral of the form
\begin{gather}\label{eq:G-defined}
G^{m,n}_{p,q} \left(z~\vline\begin{matrix}\a,\cb\\\b,\d\end{matrix} \right)
:=\frac{1}{2\pi{\rm i}}\int_{\mathcal{L}} \frac{\Gamma(1-\a - s)\Gamma(\b + s)}
{\Gamma(\cb + s)\Gamma(1-\d - s)}z^{-s}{\rm d}s,
\end{gather}
where the contour $\L$ is a simple loop that starts and ends at infinity and separates the poles of $s\to\Gamma(\b\!+\!s)$ leaving them on the left from those of $s\to\Gamma(1-\a\!-\!s)$ leaving them on the right. Details regarding the choice of the contour $\L$ and the convergence of the above integral can be found, for instance, in \cite[Section~1.1]{KilSaig}, \cite[Section~8.2]{PBM3}, \cite[Section~16.17]{NIST} and \cite[Appendix]{KPSIGMA}. If some of the vectors $\a$, $\b$, $\cb$, $\d$ are empty, they will be omitted from the above notation. The~integral in \eqref{eq:G-defined} can be evaluated by the residue theorem which, in the case when all poles of the integrand are simple, leads to a finite sum of hypergeometric functions. This expansion was derived by Meijer himself, see details in \cite[Section~4.6.2]{Slater}. Combining \cite[formula~(8.2.2.3)]{PBM3} with \cite[formula~(8.2.2.4)]{PBM3} for $0<x<1$ we can write:
\begin{gather}\label{eq:Gppexpansion1}
G^{p,0}_{p,p}\left(x\left|\begin{matrix} \cb \\ \b \end{matrix}\right.\right)
=\sum\limits_{k=1}^{p}\frac{\Gamma\big(\b_{[k]}-b_{k}\big)}{\Gamma(\cb-b_{k})}x^{b_k}\,{}_{p}F_{p-1} \left( \begin{matrix}1-\cb+b_k\\1-\b_{[k]}+b_k\end{matrix} \vline\,x \right) .
\end{gather}
The above special $G^{p,0}_{p,p}$ case of Meijer's $G$ functions plays an important role in the solution of the generalized hypergeometric differential equation, see \cite{KPSIGMA,Norlund}. It was studied in great detail by N{\o}rlund in \cite{Norlund} under different notation and without mentioning the $G$ function. As
$G^{p,0}_{p,p}(t)=0$ for $t>1$ according to \cite[Property~3]{KLJAT2017}, the Mellin transform of $G^{p,0}_{p,p}(t)$ reduces to the integral:
\begin{gather}\label{eq:MellinGp0pp}
\frac{\Gamma(\a+s)}{\Gamma(\b+s)}=\int_0^{1}t^{s-1}G^{p,0}_{p,p}
\left(t\left|\begin{matrix} \b \\ \a \end{matrix}\right.\right){\rm d}t,\qquad
\operatorname{Re}[\nu(\a;\b)]>0,\qquad \operatorname{Re}(s+\a)>0.
\end{gather}
Taking $s=k\in\N_0$ in this formula, multiplying both sides by $z^k(\sigma)_k/k!$ and summing over nonnegative integer $k$ we get the generalized Stieltjes transform representation of the function ${}_{p+1}F_{p}$ \cite[representation~(2)]{KLJAT2017}:
\begin{gather}\label{eq:Frepr}
\frac{\Gamma(\a)}{\Gamma(\b)}{}_{p+1}F_p\left. \left(\begin{matrix}\sigma,\a\\\b\end{matrix}\right\vert z\right)=\int_0^{1}G^{p,0}_{p,p}\left(t\left|\begin{matrix} \b \\ \a \end{matrix}\right.\right)
\frac{{\rm d}t}{t(1-tz)^{\sigma}},\quad\
\operatorname{Re}[\nu(\a;\b)]>0,\quad\
\operatorname{Re}(\a)>0.
\end{gather}
This representation turned out to be very useful in studying the properties of the function ${}_{p+1}F_{p}$ \cite{KLJAT2017,KPSIGMA,KPITSF2017,KPChapter2019}. In this paper we will combine it with the expansion \cite[equation~(11)]{KLJAT2017}, \cite[equation~(1.10)]{LPKITSF2020}
\begin{gather}\label{eq:Norlund}
G^{p,0}_{p,p}\!\left(\!t~\vline\begin{matrix}\b\\\a\end{matrix} \right)
=t^{a_{\omega}}(1-t)^{\nu(\a;\b)-1}
\sum\limits_{n=0}^{\infty}\frac{g_n^{p}\big(\a_{[\omega]};\b\big)}{\Gamma(\nu(\a;\b)+n)}(1-t)^n,
\end{gather}
where $\nu(\a;\b)$ is defined in \eqref{eq:notation}, found by N{\o}rlund, see \cite[equations~(1.33),~(1.35) and~(2.7)]{Norlund}. It converges in the disk $|1-t|<1$ for all complex values of parameters and each $\omega=1,2,\dots,p$. Note that if $-\nu(\a;\b)=m\in\N_0$, then the first $m+1$ terms in \eqref{eq:Norlund} vanish. The coefficients $g_n^{p}(\x;\y)$ are polynomials symmetric with respect to separate permutations of the components of $\x=(x_1,\dots,x_{p-1})$ and $\y=(y_1,\dots,y_p)$.
These coefficients serve as the key to many properties of the generalized hypergeometric function and play an important role in this paper. Their known and new properties will be discussed in detail in the following Section~\ref{section2}. In~particular, we will present expressions for these coefficients in terms of multiple hypergeometric series and in terms of generalized Bernoulli polynomials. The obvious symmetries of these coefficients lead to various identities for the terminating univariate and multivariate hypergeometric series. These facts will be summarized in Theorem \ref{th:NorlundSymmetries} (transformations of the multiple hypergeometric series) and Theorem~\ref{th:NorlundBell} (relation to the generalized Bernoulli and the complete Bell polynomials) in Section~\ref{section2} and further exemplified in Examples~\ref{example1} through~\ref{example14} in the same section.

The rest of the paper is organized as follows. In Section~\ref{section3} we give an elementary derivation of the known expansion \eqref{eq:KPSIGMA2.20} of the function $z\mapsto{}_{p+1}F_{p}(\a;\b;z)$ near $z=1$ and derive a presumably new formula for the analytic continuation of $(\a,\b)\mapsto{}_{p+1}F_{p}(\a;\b;1)$ as a function of parameters presented in Theorem~\ref{th:p+1Fp(1)-new}. We further explore some consequences of these formulas including two- and three-term Thomae's relations, two- and three-term transformations for ${}_4F_3$ with one unit shift. We further pursue this topic in Section~\ref{section4}, where we focus on transformations for ${}_{5}F_{4}(1)$ with two unit shifts in parameters. The main result of this section is Theorem~\ref{th:5F4woterm} containing a two-term transformation for such type of ${}_{5}F_{4}(1)$ series. In Section~\ref{section5} we demonstrate that the known multi-term transformations for ${}_{p+1}F_{p}(1)$ \eqref{eq:Norlund5.8} and \eqref{eq:Gppexpansion4} (which can be viewed as far-reaching generalizations of the three-term Thomae's relations) are straightforward consequences of the properties of the $G$ function. We further present another multi-term identity in Theorem~\ref{thm:Gsum1} which completes a calculation contained in the classical monograph by Slater~\cite{Slater}. The final Section~\ref{section6} is devoted to transformations of the terminating series.
Here, we give two identities for the terminating generalized hypergeometric series derived from the so-called duality relations found recently in \cite{KarpKuzn}. These results are presented in Theorems~\ref{th:alpha} and \ref{th:sumbeta}. Furthermore, we employ some recent $q$-series identities from \cite{Guo2015} to deduce three-term relations for rather general terminating hypergeometric series given in Theorems~\ref{th:Guo-general} and~\ref{th:Guo-particular}.

\section{N{\o}rlund's coefficients}\label{section2}

N{\o}rlund's coefficients $g_n^{p}(\x;\y)$ are multivariate polynomials defined by the power series gene\-ra\-ting function \eqref{eq:Norlund}.
 Another way to generate them is the inverse gamma series \cite[formula~(2.21)]{Norlund}
\begin{gather*}
\frac{\Gamma(z+\x)}{\Gamma(z+\y)}=\sum\limits_{n=0}^{\infty}
\frac{g_n^{p}(\x;\y)}{\Gamma(z+\nu(\x;\y)+n)}
\end{gather*}
obtained by substituting \eqref{eq:Norlund} into \eqref{eq:MellinGp0pp} and integrating term-wise. This formula implies the symmetry with respect to separate permutations of the components of $\x$ and $\y$ and the shift invariance $g_n^{p}(\x+\alpha;\y+\alpha)=g_n^{p}(\x;\y)$ for any $\alpha$ \cite[p.~5]{LPKITSF2020}. The latter means that $g_n^{p}$ is defined on the $2p-2$ (complex-) dimensional factor space $\C^{2p-1}\slash\C$ (the factor is taken with respect to addition). In other words, we can view it as a function of $2p-2$ variables $x_2-x_1,\dots,x_{p-1}-x_1$, $y_1-x_1,\dots,y_p-x_1$. The polynomials $g_n^{p}(\x;\y)$ satisfy the following recurrence relation with respect to $p$ \cite[formula~(2.7)]{Norlund}:
\begin{gather}\label{eq:Norlundcoeff}
g_n^{p}(\x;\y)=\sum\limits_{m=0}^{n}\frac{(y_{p}-x_{p-1})_{n-m}}{(n-m)!}
(\nu_{p-1}+m)_{n-m}g_m^{p-1}\big(\x_{[p-1]};\y_{[p]}\big),\qquad
p=2,3,\dots,
\end{gather}
where $\nu_{p-1}=\nu\big(\x;\y_{[p]}\big)$ (see \eqref{eq:notation}) and the initial values are $g_0^{1}(-;y_1)=1$, $g_n^{1}(-;y_1)=0$, $n\ge1$. For $p=2$ immediately from this recurrence we get
\begin{gather}\label{eq:Norlund-p2}
g_n^{2}(x;\y)=\frac{(y_2-x)_{n}(y_1-x)_{n}}{n!}.
\end{gather}
Recurrence \eqref{eq:Norlundcoeff} can also be solved in general to yield \cite[formula~(2.11)]{Norlund}:
\begin{gather}\label{eq:Norlund-explicit}
g_n^{p}(\x;\y)=\sum\limits_{0\leq{j_{1}}\leq{j_{2}}\leq\cdots\leq{j_{p-2}}\leq{n}}
\prod\limits_{m=1}^{p-1}\frac{(\nu_{m}+j_{m-1})_{j_{m}-j_{m-1}}}
{(j_{m}-j_{m-1})!}(y_{m+1}-x_{m})_{j_{m}-j_{m-1}},
\end{gather}
where $j_0=0$, $j_{p-1}=n$ and
\begin{gather}\label{eq:nu-m}
\nu_m=\sum_{k=1}^{m}(y_{k}-x_{k}).
\end{gather}
Note that formula \eqref{eq:Norlund-explicit} implies that $g_n^{p}(\x;\y)\ge0$ if
$\nu_m\ge0$ and $y_{m+1}\ge x_{m}$ for $m=1,\dots,p-1$ or, equivalently, if $\nu_m\ge0$ and $\nu_m-y_{m}+y_{m+1}\ge0$ for $m=1,\dots,p-1$.

Introduce the new summation indices according to
\begin{gather*}
j_1\to l_1,\qquad
j_2-j_1\to l_2,\qquad
\dots,\qquad
 j_{p-2}-j_{p-3}\to l_{p-2}.
\end{gather*}
Then, applying the relations
\begin{gather}\label{eq:pochammer}
(\alpha)_{n-s}=\frac{(-1)^s(\alpha)_n}{(1-\alpha-n)_s},\qquad
(n-s)!=\frac{(-1)^sn!}{(-n)_s},\qquad
(\alpha+s)_{n-s}=\frac{(\alpha)_n}{(\alpha)_s},
\end{gather}
formula \eqref{eq:Norlund-explicit} can be written as $(p-2)$--fold terminating hypergeometric series
\begin{gather}
g_n^{p}(\x;\y)=\frac{(\nu_{p-1})_{n}(y_{p}-x_{p-1})_{n}}{n!}
\nonumber
\\ \hphantom{g_n^{p}(\x;\y)=}
{}\times\sum\limits_{\l_{p-2}\in\N^{p-2}_{0}}\frac{(-n)_{|\l_{p-2}|}}{(1-y_{p}+x_{p-1}-n)_{|\l_{p-2}|}}\prod_{s=1}^{p-2}\frac{(\nu_{s})_{|\l_{s}|}(y_{s+1}-x_s)_{l_{s}}}{(\nu_{s+1})_{|\l_{s}|}l_s!},
\label{eq:Norlund-explicit1}
\end{gather}
where $\l_s=(l_1,\dots,l_s)$, $|\l_s|=l_1+\cdots+l_s$, valid for $p\ge3$. We were unable to find any name attached to this particular type of series in the generality given above. Nevertheless, an anonymous referee drew our attention to an identity for $q$-hypergeometric series due to George Andrews contained in \cite[Theorem~4]{Andrews}. By taking the limit $q\to1$ and appropriately changing notation, his result leads to the following relation:
\begin{gather}
\sum\limits_{\l_{p-2}\in\N^{p-2}_{0}}\frac{(-n)_{|\l_{p-2}|}}{(1-y_{p}+x_{p-1}-n)_{|\l_{p-2}|}}
\prod_{s=1}^{p-2}\frac{(x_{s+1})_{|\l_{s}|}(y_{s+1})_{|\l_{s}|}
(\mu_p-x_s-y_s)_{l_{s}}}{(\mu_p-x_s)_{|\l_{s}|}(\mu_p-y_s)_{|\l_{s}|}l_{s}!}\nonumber
\\ \hphantom{\sum\limits_{\l_{p-2}\in\N^{p-2}_{0}}}
{}=\frac{(\mu_p-x_{p-1})_n(y_p)_n}{(\mu_p)_n(y_p-x_{p-1})_n}\, {}_{2p+1}F_{2p}\!
\left( \begin{matrix}-n,\x,\y_{[p]},(\mu_p+1)/2,\mu_p-1 \\n+\mu_p,\mu_p-\x,\mu_p-\y_{[p]}, (\mu_p-1)/2\end{matrix} \right)\!,
\label{eq:Andrews}
\end{gather}
where as before $\x\in \C^{p-1}$, $\y\in\C^{p}$, $\l_s=(l_1,\dots,l_s)$, $|\l_s|=l_1+\cdots+l_s$ and $\mu_p=y_{p}+y_{p-1}$ for brevity. This surprising formula expresses a multiple series somewhat similar to \eqref{eq:Norlund-explicit1} in terms of very-well poised ${}_{2p+1}F_{2p}(1)$. Nevertheless, to the best of our understanding the left hand side of Andrew's identity is essentially different from the right-hand side of \eqref{eq:Norlund-explicit1}. To justify this claim let us compare them for $p=3$. Then, the left-hand side of \eqref{eq:Andrews}
takes the form (recall that omitted argument equals $1$):
\begin{gather*}
{}_4F_3\left( \begin{matrix}-n,x_{2},y_{2},y_{2}+y_{3}-y_{1}-x_{1}
\\1-y_{3}+x_{2}-n,y_{2}+y_{3}-x_{1},y_{2}+y_{3}-y_{1}\end{matrix} \right)\!,
\end{gather*}
while formula \eqref{eq:Norlund-explicit1} reads
\begin{gather}\label{eq:gnp3}
g_n^{3}(\x;\y)=\frac{(\nu_{2})_n(y_{3}-x_{2})_n}{n!}\, {}_3F_2
\left( \begin{matrix}-n,y_{1}-x_{1},y_{2}-x_{1}\\1-y_{3}+x_{2}-n,\nu_{2}\end{matrix} \right)\!,
\end{gather}
where $\nu_2$ is defined in \eqref{eq:nu-m}. The first expression represents the general Saalsch\"{u}tzian ${}_4F_3$, while the second is the general terminating ${}_3F_2$. As there is no (known) transformation mapping one into the other, these two expressions are likely to be essentially different.

For $p=4$, i.e., $\x=(x_1,x_2,x_3)$, $\y=(y_1,y_2,y_3,y_4)$, we have by~\eqref{eq:Norlundcoeff} and in view of~\eqref{eq:pochammer},
\begin{gather*}
g_n^{4}(\x;\y)=
\frac{(y_4-x_3)_{n}(\nu_3)_{n}}{n!}\sum\limits_{l=0}^{n} \frac{(-n)_{l}(y_3-x_2)_{l}(\nu_2)_{l}}{(1-y_4+x_3-n)_{l}(\nu_3)_{l}l!}
\,{}_3F_2 \left( \begin{matrix}-l,y_{1}-x_{1},y_{2}-x_{1}\\1-y_{3}+x_{2}-l,\nu_2\end{matrix} \right)\!.
\end{gather*}
Alternatively, \eqref{eq:Norlund-explicit1} yields
\begin{gather}\label{eq:KampedeFeriet1}
g_n^{4}(\x;\y)=\frac{(y_4-x_3)_{n}(\nu_3)_{n}}{n!}\, F^{2:2:1}_{2:1:0}
\!\left(\left. \begin{matrix}-n,\nu_2:\nu_1,y_2-x_1:y_3-x_2
\\1-y_4+x_3-n,\nu_3:\nu_2:-\end{matrix}\right|1,1\!\right)\!,
\end{gather}
where
$F^{p:q:r}_{s:t:u}$ denotes the Kamp\'e de F\'eriet function defined by \cite[Section~1.3, formula~(28)]{SKbook}:
\begin{gather*}
F^{p:q:r}_{s:t:u}\!\left(\left. \begin{matrix}\a:\cb:\e\\\b:\d:\f\end{matrix}\right|z,w\!\right)
=\sum_{k,l\ge0}\frac{(\a)_{k+l}(\cb)_{k}(\e)_{l}}{(\b)_{k+l}(\d)_{k}(\f)_{l}}\frac{z^kw^l}{k!l!},
\end{gather*}
where $\a\in\C^{p}$, $\b\in\C^{s}$, $\cb\in\C^{q}$, $\d\in\C^{t}$, $\e\in\C^{r}$, $\f\in\C^{u}$ and $z$, $w$ can take arbitrary values as long as the above series terminates. In the general case, convergence conditions are given in \cite[Chapter~14]{NguyenYakub}, while the analytic continuation can be constructed following the ideas from \cite{Bezrodnykh}.

For $p=5$, i.e., $\x\in\C^{p-1}$, $\y\in\C^{p}$, by \eqref{eq:Norlundcoeff} we obtain
\begin{gather*}
g_n^{5}(\x;\y)=\frac{(\nu_{4})_{n}(y_{5}-x_{4})_{n}}{n!}\sum\limits_{m=0}^{n}
\frac{(-n)_{m}(y_{4}-x_{3})_{m}(\nu_3)_{m}}{(1-y_{5}+x_{4}-n)_{m}(\nu_{4})_{m}m!}
\\ \hphantom{g_n^{5}(\x;\y)=}
{}\times\sum\limits_{l=0}^{m}\frac{(-m)_{l}(y_{3}-x_{2})_{l}
(\nu_{2})_{l}}{(1-y_{4}+x_{3}-n)_{l}(\nu_{3})_{l}l!}\,
{}_3F_2\! \left( \begin{matrix}-l,y_{1}-x_{1},y_{2}-x_{1}\\1-y_{3}+x_{2}-l,\nu_2
\end{matrix} \right)\!,
\end{gather*}
which, according to \eqref{eq:Norlund-explicit1}, is equal to
\begin{gather}
g_n^{5}(\x;\y)=\frac{(\nu_{4})_{n}(y_{5}-x_{4})_{n}}{n!}\nonumber
\\ \hphantom{g_n^{5}(\x;\y)=}
{}\times F^{(3)}\left[
\left. \begin{matrix}-n,\nu_3::\nu_2;-;-:\nu_1,y_2-x_1;y_3-x_2;y_4-x_3
\\1-y_5+x_4-n,\nu_4::\nu_3;-;-:\nu_2;-;-\end{matrix}\right|1,1,1\right]\!.
\label{eq:SrivF3-1}
\end{gather}
Here $F^{(3)}$ is Srivastava's hypergeometric function of three variables defined by \cite[Section~1.5, formula~(14)]{SKbook}:
\begin{gather*}
F^{(3)}\!\left[
\left. \begin{matrix}\a::\cb;\cb';\cb'':\e;\e';\e''\\\b::\d;\d';\d'':\f;\f';\f''
\end{matrix}\right|z,w,u\right]
\\ \qquad
{}=\sum\limits_{n,k,l\ge0}\frac{(\a)_{n+k+l}(\cb)_{n+k}(\cb')_{k+l}(\cb'')_{n+l}(\e)_{n}(\e')_{k}(\e'')_{l}}
{(\b)_{n+k+l}(\d)_{n+k}(\d')_{k+l}(\d'')_{n+l}(\f)_{n}(\f')_{k}(\f'')_{l}}\frac{z^nw^ku^l}{n!k!l!}
\end{gather*}
with parameter vectors of arbitrary finite sizes and arbitrary values of variables as long as the above series terminates.

Another recurrence relation found by certain renaming in N{\o}rlund's formula \cite[formu\-la~(1.41)]{Norlund} is given by
\begin{gather}\label{eq:NorlCoefidentity}
g_n^{p}(\x;\y)=\frac{(\nu(\x;\y)-y_{p-1})_n(\nu(\x;\y)-y_{p})_{n}}{n!}
\sum_{s=0}^{n}\frac{(-n)_{s}g_{s}^{p-1}\big(1-\y_{[p-1,p]};1-\x\big)}
{(\nu(\x;\y)-y_{p-1})_s(\nu(\x;\y)-y_{p})_{s}}
\end{gather}
with $p\ge3$ and initial values \eqref{eq:Norlund-p2}. Using this recurrence for $p=3$ we get (see also \cite[Pro\-perty~6]{KLJAT2017}):
\begin{gather}\label{eq:gnp3N}
g_n^{3}(\x;\y)=\frac{(\nu(\x;\y)-y_2)_n(\nu(\x;\y)-y_3)_n}{n!}\,
{}_3F_2\!\left( \begin{matrix}-n,y_1-x_{1},y_1-x_{2}\\\nu(\x;\y)-y_2,\nu(\x;\y)-y_3\end{matrix} \right)\!.
\end{gather}
For $p=4$, denoting $\psi_m=\sum_{j=1}^{m}y_{j}-\sum_{j=1}^{m-1}x_{j}=\nu_m+x_m$ we have (see also \cite[Property~6]{KLJAT2017}):
\begin{gather}
g_n^{4}(\x;\y)= \frac{(\psi_4-y_3)_n(\psi_4-y_4)_n}{n!}
\nonumber
\\ \hphantom{g_n^{4}(\x;\y)=}
{}\times\sum\limits_{k=0}^{n}\frac{(-n)_k(\psi_2-x_{2})_k(\psi_2-x_{3})_k}{k!(\psi_4-y_3)_k(\psi_4-y_4)_k}\,
{}_{3} F_{2}\!\left( \begin{matrix}-k,y_1-x_{1},y_2-x_{1}\\ \psi_2-x_{2}, \psi_2-x_{3}\end{matrix} \right)\!.\label{eq:nor}
\end{gather}

The recurrence \eqref{eq:NorlCoefidentity} can also be solved in general.
For odd $p\ge3$ we obtain\vspace{-.5ex}
\begin{gather*}
g_n^{p}(\x;\y)=\sum\limits_{0\leq{j_{1}}\leq{j_{2}}\leq\cdots\leq{j_{p-2}}\leq{j_{p-1}}}
\frac{(-j_{p-1})_{j_{p-2}}(-j_{p-2})_{j_{p-3}}\cdots(-j_{2})_{j_{1}}}{j_{p-1}!j_{p-2}!\cdots j_1!}
\\ \hphantom{g_n^{p}(\x;\y)}
{}\times\!\!\!\!\!\prod\limits_{\stackrel{m=2}{m~\text{is even}}}^{p-1}\!\!\!\frac{(\psi_{m-1}-x_{m-1})_{j_{m-1}}(\psi_{m-1}-x_{m})_{j_{m-1}}(\psi_{m+1}-y_{m})_{j_{m}}
(\psi_{m+1}-y_{m+1})_{j_{m}}}
{(\psi_{m-1}\!-x_{m-1})_{j_{m-2}}(\psi_{m-1}\!-x_{m})_{j_{m-2}}
(\psi_{m+1}-y_{m})_{j_{m-1}}(\psi_{m+1}-y_{m+1})_{j_{m-1}}},
\end{gather*}
where $j_0=0$, $j_{p-1}=n$. Similarly, for even $p\ge4$, we get
\begin{gather*}
g_n^{p}(\x;\y)=\sum\limits_{0\leq{j_{1}}\leq{j_{2}}\leq\cdots\leq{j_{p-2}}
\leq{j_{p-1}}}\frac{(-j_{p-1})_{j_{p-2}}(-j_{p-2})_{j_{p-3}}
\cdots(-j_{2})_{j_{1}}}{j_{p-1}!j_{p-2}!\cdots j_1!}(\psi_2-y_2)_{j_1}(\psi_2-y_1)_{j_1}
\\ \hphantom{g_n^{p}(\x;\y)=}
{}\times\prod\limits_{\stackrel{m=2}{m~\text{is even}}}^{p-2}\!\!
\frac{(\psi_{m}\!-x_{m})_{j_{m}}(\psi_{m}\!-x_{m+1})_{j_{m}}(\psi_{m+2}\!-y_{m+1})_{j_{m+1}}
(\psi_{m+2}\!-y_{m+2})_{j_{m+1}}}
{(\psi_{m}-x_{m})_{j_{m-1}}(\psi_{m}-x_{m+1})_{j_{m-1}}(\psi_{m+2}-y_{m+1})_{j_{m}}
(\psi_{m+2}-y_{m+2})_{j_{m}}}.
\end{gather*}
Introduce the new summation indices according to
\begin{gather*}
j_1\to l_1,\qquad j_2-j_1\to l_2,\qquad
\dots, \qquad
j_{p-2}-j_{p-3}\to l_{p-2}.
\end{gather*}
Then, applying the relations \eqref{eq:pochammer} for odd $p\ge3$ we get the multiple hypergeometric representations ($\psi_m=\nu_m+x_m$):
\begin{subequations}\label{eq:NorlundMultiple}
\begin{gather}\label{eq:multiple_podd}
g_n^{p}(\x;\y)=\sum\limits_{\l_{p-2}\in\N^{p-2}_{0}}\frac{(-n)_{|\l_{p-2}|}(-1)^{l_2}(-1)^{l_4}
\cdots(-1)^{l_{p-3}}}{l_{1}!l_{2}!\cdots l_{p-2}!n!}
\\ \qquad
{}\times\!\!\!\prod\limits_{\stackrel{m=2}{m~\text{is even}}}^{p-1}\!\!\!\!\frac{(\psi_{m-1}-x_{m-1})_{|\l_{m-1}|}
(\psi_{m-1}-x_{m})_{|\l_{m-1}|}(\psi_{m+1}-y_{m})_{|\l_{m}|}(\psi_{m+1}-y_{m+1})_{|\l_{m}|}}
{(\psi_{m-1}\!-x_{m-1})_{|\l_{m-2}|}(\psi_{m-1}-x_{m})_{|\l_{m-2}|}
(\psi_{m+1}\!-y_{m})_{|\l_{m-1}|}(\psi_{m+1}\!-y_{m+1})_{|\l_{m-1}|}},\nonumber
\end{gather}
where $l_0=0$, $l_{p-1}=n-l_1-l_2-\cdots-l_{p-2}$, $\l_m=(l_1,\dots,l_{m})$ and $|\l_m|=l_1+\cdots+l_{m}$. In~a~si\-mi\-lar fashion, for even $p\ge4$ we get
\begin{gather}\label{eq:multiple_peven}
g_n^{p}(\x;\y)=\sum\limits_{\l_{p-2}\in\N^{p-2}_{0}}
\frac{(-n)_{|\l_{p-2}|}(-1)^{l_1}(-1)^{l_3}\cdots(-1)^{l_{p-3}}}{l_{1}!l_{2}!\cdots l_{p-2}!n!}(\psi_2-y_2)_{l_1}(\psi_2-y_1)_{l_1}
\\ \hphantom{g_n^{p}(\x;\y)=}
{}\times\!\!\!\!\prod\limits_{\stackrel{m=2}{m~\text{is even}}}^{p-2}\!\!\!\!\frac{(\psi_{m}-x_{m})_{|\l_{m}|}(\psi_{m}-x_{m+1})_{|\l_{m}|}
(\psi_{m+2}-y_{m+1})_{\l_{m+1}}(\psi_{m+2}-y_{m+2})_{|\l_{m+1}|}}
{(\psi_{m}-x_{m})_{|\l_{m-1}|}(\psi_{m}-x_{m+1})_{|\l_{m-1}|}
(\psi_{m+2}\!-y_{m+1})_{|\l_{m}|}(\psi_{m+2}\!-y_{m+2})_{|\l_{m}|}}.\nonumber
\end{gather}
\end{subequations}
For $p=4$ this yields (returning to $\nu_m=\psi_m-x_m$):\vspace{-.5ex}
\begin{gather}
g_n^{4}(\x;\y)=\frac{(\nu_3)_{n}(\nu_3-y_3+y_4)_{n}}{n!}\nonumber
\\ \hphantom{g_n^{4}(\x;\y)=}
{}\times F^{3:2:0}_{2:2:0}\!\left(\left. \begin{matrix}-n,\nu_2,\nu_2+x_2-x_3:\nu_1,y_2-x_1:-\\\nu_3,\nu_3+y_4-y_3:\nu_2,\nu_2+x_2-x_3:-\end{matrix}\right|-1,1\!\right)\!.
\label{eq:KampedeFeriet2}
\end{gather}

For $p=5$ we get the following expression in terms of Srivastava's $F^{(3)}$ \cite[Section 1.5, formula~(14)]{SKbook}:
\begin{gather}\label{eq:SrivF3-2}
g_n^{5}(\x;\y)=\frac{(\nu_{4})_{n}(\nu_{4}+y_{5}-y_{4})_{n}}{n!}
\\ \hphantom{g_n^{5}(\x;\y)}
{}\times F^{(3)}\!\left[\left. \begin{matrix}-n,\nu_{3},\nu_{3}+x_{3}-x_{4}::\nu_{2},
\nu_{2}+y_{3}-y_{2};-;-:\nu_{1},y_{1}-x_{2};-;-
\\
\nu_4,\nu_{4}+y_{5}-y_{4}::\nu_3,\nu_3+x_{3}-x_{4};-;-:\nu_2,\nu_{2}+y_{3}-y_{2};-;-
\end{matrix}\right|1,-1,1\right]\!.\nonumber
\end{gather}

Another symmetry of the N{\o}rlund's coefficients comes from the observation mentioned above: setting $\nu(\a;\b)=-m$, $m\in \N_0$, in \eqref{eq:Norlund} we obtain
\begin{gather}\label{eq:NorlundSpecial}
G^{p,0}_{p,p}\!\left(\!t~\vline\begin{matrix}\b\\\a\end{matrix} \right)=t^{a_{\omega}}
\sum\limits_{n=0}^{\infty}\frac{g_{n+m+1}^{p}\big(\a_{[\omega]};\b\big)}{n!}(1-t)^n,
\end{gather}
where $\omega\in\{1,\dots,p\}$. In particular, putting $t=1$ we get $g_{m+1}^{p}\big(\a_{[\omega]};\b\big)$ on the right-hand side which implies that
\begin{gather*}
g_{m+1}^{p}\big(\a_{[\omega_1]};\b\big)=g_{m+1}^{p}\big(\a_{[\omega_2]};\b\big),\qquad
\forall\omega_{1},\omega_{2}\in\{1,\dots,p\},
\end{gather*}
as long as $\sum_{k=1}^{p}(b_k-a_k)=-m$. Equivalently,
\begin{gather}\label{eq:NorlundReflection}
g_{m+1}^{p}(\x;\y)=g_{m+1}^{p}\big(\big(\x_{[\omega]},\nu(\x;\y)+m\big);\y\big)
\end{gather}
for arbitrary $\omega\in\{1,\dots,p-1\}$.
The main findings of this section so far can be summarized as follows.
\begin{Theorem}\label{th:NorlundSymmetries}
	The terminating multiple hypergeometric series defined in \eqref{eq:Norlund-explicit1} and \eqref{eq:NorlundMultiple} are equal, symmetric with respect to permutations of the components of $\x$, symmetric with respect to permutations of the components of $\y$, satisfy the transformation formula \eqref{eq:NorlundReflection} and are non-negative if
	$\nu_m\ge0$ and $y_{m+1}\ge x_{m}$ for $m=1,\dots,p-1$ or, equivalently, if $\nu_m\ge0$ and $\nu_m-y_{m}+y_{m+1}\ge0$ for $m=1,\dots,p-1$.
\end{Theorem}

\begin{Example}\label{example1} Equating \eqref{eq:gnp3} to \eqref{eq:gnp3N} and setting $y_1-x_1=\alpha_1$, $y_1-x_{2}=\alpha_2$, $y_1+y_2-x_1\allowbreak-x_2=\beta_1$, $y_1+y_3-x_1-x_2=\beta_2$ we recover
Sheppard's transformation \cite[Corollary 3.3.4]{AAR}, \cite[Appendix, formula (I)]{RJRJR}
\begin{gather}\label{eq:Sheppard}
{}_3F_2\!\left( \begin{matrix}-n,\alpha_1,\alpha_2\\\beta_1,\beta_2\end{matrix} \right)
=\frac{(\beta_2-\alpha_1)_n}{(\beta_2)_n}\, {}_3F_2\!\left( \begin{matrix}-n,\alpha_1,\beta_1-\alpha_2\\\beta_1,1-\beta_2+\alpha_1-n\end{matrix} \right)
\end{gather}
also rediscovered by B\"{u}hring \cite[formula~(4.1)]{Buehring95}. Note that $\alpha_1,\alpha_2,\beta_1,\beta_2\ge0$ implies nonnegativity of both sides of \eqref{eq:Sheppard}.
\end{Example}

\begin{Example}%\label{example2}
The symmetry $g_n^{3}((x_1,x_2);\y)=g_n^{3}((x_2,x_1);\y)$ applied to \eqref{eq:gnp3} after the appropriate change of notation yields
\begin{gather*}
{}_3F_2\!\left( \begin{matrix}-n,\alpha_1,\alpha_2\\\beta_1,\beta_2\end{matrix} \right)
=(-1)^n\frac{(1-s)_n}{(\beta_1)_n}\,{}_3F_2\!
\left( \begin{matrix}-n,\beta_2-\alpha_1,\beta_2-\alpha_2\\s-n,\beta_2\end{matrix} \right)\!, \\
s=\beta_1+\beta_2-\alpha_1-\alpha_2+n,
\end{gather*}
which is precisely the transformation \cite[Appendix, formula~(II)]{RJRJR}. The symmetry
\[
g_n^{3}(\x;(y_1,y_2,y_3))=g_n^{3}(\x;(y_3,y_2,y_1))
\]
leads to
\begin{gather*}
{}_3F_2\!\left( \begin{matrix}-n,\alpha_1,\alpha_2\\\beta_1,\beta_2\end{matrix} \right)
=\frac{(\beta_1-\alpha_1)_n(\beta_2-\alpha_1)_n}{(\beta_1)_n(\beta_2)_n}\, {}_3F_2\!\left( \begin{matrix}-n,\alpha_1,1-s\\1+\alpha_1-\beta_1-n,1+\alpha_1-\beta_2-n\end{matrix} \right)
\end{gather*}
which coincides with \cite[Appendix, formula~(III)]{RJRJR}.
\end{Example}

\begin{Example}%\label{example3}
The symmetry
\begin{gather*}
g_n^{4}((x_1,x_2,x_3);\y)=g_n^{4}((x_3,x_2,x_1);\y)
\end{gather*}
applied to \eqref{eq:KampedeFeriet1} leads, after change of notation, to
\begin{gather*}
\frac{(\mu)_{n}}{(\alpha+\beta+\gamma+\mu-\eta)_{n}}\, F^{2:2:1}_{2:1:0}\!
\left(\left. \begin{matrix}-n,\delta:\alpha,\beta:\gamma\\1-\mu-n,\eta:\delta:-
\end{matrix}\right|1,1\!\right)
\\ \qquad
{}=F^{2:2:1}_{2:1:0}\!\left(\left. \begin{matrix}-n,\delta+\eta-\alpha-\beta-\gamma:
\eta-\beta-\gamma,\eta-\alpha-\gamma:\gamma\\1+\eta-\alpha-\beta-\gamma-\mu-n,\eta:
\delta+\eta-\alpha-\beta-\gamma:-\end{matrix}\right|1,1\!\right)\!.
\end{gather*}
Note that $\alpha,\beta,\gamma,\mu,\eta,\delta\ge0$ implies that both sides of the above formula have the same sign as $(\alpha+\beta+\gamma+\mu-\eta)_{n}$.
\end{Example}

\begin{Example}%\label{example4}
The symmetry from the previous example applied to \eqref{eq:KampedeFeriet2} yields
\begin{gather*}
F^{3:2:0}_{2:2:0}\!\left(\left. \begin{matrix}-n,\alpha,\beta:\gamma,\delta:-
\\\mu,\eta:\alpha,\beta:-\end{matrix}\right|-1,1\!\right)
\\ \qquad
{}=F^{3:2:0}_{2:2:0}\!\left(\left. \begin{matrix}-n,\alpha+\beta-\delta-\gamma,\beta:
\beta-\delta,\beta-\gamma:-\\\mu,\eta:\alpha+\beta-\delta-\gamma,\beta:-\end{matrix}\right|-1,1\!\right)\!.
\end{gather*}
Both sides of this formula are non-negative if $\gamma,\delta,\mu\ge0$, $\eta\ge\alpha\ge0$ and $\beta\le\mu$.
\end{Example}

\begin{Example}%\label{example5}
The symmetry
\begin{gather*}
g_n^{4}(\x;(y_1,y_2,y_3,y_4))=g_n^{4}(\x;(y_4,y_2,y_3,y_1))
\end{gather*}
applied to \eqref{eq:KampedeFeriet1} yields, after change of notation,
\begin{gather*}
\frac{(\mu)_{n}(\eta)_{n}}{(\eta-\beta-\gamma)_{n}(\beta+\gamma+\mu)_{n}}\,F^{2:2:1}_{2:1:0}\!
\left(\left. \begin{matrix}-n,\delta:\alpha,\beta:\gamma\\1-\mu-n,\eta:\delta:-
\end{matrix}\right|1,1\!\right)
\\ \qquad
{}=F^{2:2:1}_{2:1:0}\!\left(\left. \begin{matrix}-n,\delta+\beta+\gamma+\mu-\eta:
\alpha+\beta+\gamma+\mu-\eta,\beta:\gamma\\1+\beta+\gamma-\eta-n,\beta+\gamma+\mu:
\delta+\beta+\gamma+\mu-\eta:-\end{matrix}\right|1,1\!\right)\!.
\end{gather*}
Note that for $\alpha,\beta,\gamma,\mu,\eta,\delta\ge0$ both sides of the above formula have the same sign as $(\eta-\beta-\gamma)_{n}$.
\end{Example}

\begin{Example}%\label{example6}
The symmetry from the previous example applied to \eqref{eq:KampedeFeriet2} yields
\begin{gather*}
\frac{(\mu)_{n}}{(\delta+\mu+\eta-\alpha-\beta)_{n}}F^{3:2:0}_{2:2:0}\!
\left(\left. \begin{matrix}-n,\alpha,\beta:\gamma,\delta:-\\\mu,\eta:\alpha,\beta:-
\end{matrix}\right|-1,1\!\right)
\\ \qquad
{}=F^{3:2:0}_{2:2:0}\!\left(\left. \begin{matrix}-n,\delta+\eta-\beta, \delta+\eta-\alpha:\delta+\gamma+\eta-\alpha-\beta, \delta:-\\\delta+\mu+\eta-\alpha-\beta, \eta:\delta+\eta-\beta, \delta+\eta-\alpha:-\end{matrix}\right|-1,1\!\right)\!.
\end{gather*}
Note that for $\gamma,\delta,\mu\ge0$, $\eta\ge\alpha\ge0$ and $\beta\le\mu$ both sides of the above formula have the same sign as $(\delta+\mu+\eta-\alpha-\beta)_{n}$.
\end{Example}

\begin{Example}%\label{example7}
If we take $p=3$ in \eqref{eq:NorlundReflection} and use \eqref{eq:gnp3N} for $g_n^3$ we again arrive at the second relation in Example~2 equivalent to \cite[Appendix, formula~(III)]{RJRJR}.
\end{Example}

\begin{Example}%\label{example8}
For $p=4$ application of \eqref{eq:NorlundReflection} with $x_3\to\nu_3+y_4+m$ to
\eqref{eq:KampedeFeriet2} gives the following two-term transformation:
\begin{gather*}
F^{3:2:0}_{2:2:0}\!\left(\left. \begin{matrix}-m-1,\alpha,\beta-m:\gamma,\delta:-
\\-\mu-m,-\eta-m:\alpha,\beta-m:-\end{matrix}\right|-1,1\!\right)
\\ \qquad
{}=\frac{(\alpha+\mu)_{m+1}(\alpha+\eta)_{m+1}}{(\mu)_{m+1}(\eta)_{m+1}}\,F^{3:2:0}_{2:2:0}\!
\left(\left. \begin{matrix}-m-1,\alpha,\alpha+\beta+\eta+\mu:\gamma,\delta:-
\\\alpha+\eta,\alpha+\mu:\alpha,\alpha+\beta+\eta+\mu:-\end{matrix}\right|-1,1\!\right)\!.
\end{gather*}
Both sides of the above identity have the sign of $(\eta)_{m+1}$ if $\alpha,\delta,\gamma,\mu,\alpha+\eta\ge0$ and $\beta+\mu\le0$.
\end{Example}

\begin{Example}%\label{example9}
For $p=5$ application of \eqref{eq:NorlundReflection} with $x_4\to\nu_4+y_4+m$ to \eqref{eq:SrivF3-2} yields
\begin{gather*}
F^{(3)}\!\left[\left. \begin{matrix}-m-1,\alpha,\beta-m::\gamma,\delta;-;-:\mu,\eta;-;-
\\
-\sigma-m,-\zeta-m::\alpha,\beta-m;-;-:\gamma,\delta;-;-\end{matrix}\right|1,-1,1\right]
\\ \qquad
{}=\frac{(\alpha+\sigma)_{m+1}(\alpha+\zeta)_{m+1}}{(\sigma)_{m+1}(\zeta)_{m+1}}
\\ \qquad\hphantom{=}
{}\times F^{(3)}\!\left[\left. \begin{matrix}-m-1,\alpha,\alpha+\beta+\sigma+\zeta::
\gamma,\delta;-;-:\mu,\eta;-;-
\\
\alpha+\zeta,\alpha+\sigma::\alpha,\alpha+\beta+\sigma+\zeta;-;-:\gamma,\delta;-;-
\end{matrix}\right|1,-1,1\right]\!.
\end{gather*}
If $\alpha,\gamma,\mu,\sigma,\alpha+\zeta\ge0$, $\delta\ge\mu$, $\eta\le\gamma$ and $\beta+\sigma\le0$, then both sides of the above identity have the same sign as $(\zeta)_{m+1}$.
\end{Example}

\begin{Example}%\label{example10}
The equality right-hand side of \eqref{eq:KampedeFeriet1} = right-hand side of \eqref{eq:KampedeFeriet2} represents a presumably new transformation for the terminating Kamp\'e de F\'eriet function:
\begin{gather*}
F^{2:2:1}_{2:1:0}\!\left(\left. \begin{matrix}-n,\alpha:\beta,\gamma:\delta\\1-\mu-n,\eta:\alpha:-\end{matrix}\right|1,1\!\right)
=\frac{(\alpha+\mu)_{n}}{(\mu)_{n}}\, F^{3:2:0}_{2:2:0}\!\left(\left. \begin{matrix}-n,\alpha,\eta-\delta:\beta,\gamma:-\\\eta,\alpha+\mu:\alpha,\eta-\delta:-\end{matrix}\right|-1,1\!\right)\!.
\end{gather*}
The condition $\alpha,\beta,\gamma,\mu,\eta,\delta\ge0$ ensures that both sides of the above formula are non-negative.
\end{Example}

\begin{Example}\label{example11} Similarly, equating the right-hand side of \eqref{eq:SrivF3-1} to the right-hand side of \eqref{eq:SrivF3-2} we obtain a presumably new transformation for the terminating Srivastava's $F^{(3)}$ function:
\begin{gather*}
\frac{(\sigma)_{n}}{(\alpha+\sigma)_{n}}F^{(3)}\!\left[
\left. \begin{matrix}-n,\alpha::\beta;-;-:\gamma,\delta;\mu;\eta\\1-\sigma-n,\zeta::\alpha;-;-:\beta;-;-\end{matrix}\right|1,1,1\right]
\\ \qquad
{}=F^{(3)}\!\left[\left. \begin{matrix}-n,\alpha,\zeta-\eta::\beta,\gamma+\mu;-;-:\gamma,\beta-\delta;-;-
\\
\zeta,\alpha+\sigma::\alpha,\zeta-\eta;-;-:\beta,\gamma+\mu;-;-\end{matrix}\right|1,-1,1\right]\!.
\end{gather*}
The condition $\alpha,\beta,\gamma,\mu,\eta,\delta,\sigma,\zeta\ge0$ ensures that both sides of the above formula are non-negative.
\end{Example}

Let us remark in passing that a number of transformations for the terminating hypergeometric and Kamp\'e de F\'eriet functions were motivated by the study of the symmetries of the $3-j$, $6-j$ and $9-j$ coefficient appearing in the quantum mechanical treatment of angular momentum, see~\cite{RJRJR,JeugtPitreRao} and references therein for details.

Next, recall that the Bernoulli--N{\o}rlund (or the generalized Bernoulli) polynomial $\Be^{(\sigma)}_{k}(x)$ is defined by the generating function \cite[formula~(1)]{Norlund61}:
\begin{gather}\label{eq:Bern-Norl-defined}
\frac{t^{\sigma}{\rm e}^{zt}}{(e^t-1)^{\sigma}}=\sum\limits_{k=0}^{\infty}\Be^{(\sigma)}_{k}(z)\frac{t^k}{k!}.
\end{gather}
In particular, $\Be^{(1)}_{k}(z)=\Be_{k}(z)$ is the classical Bernoulli polynomial. In \cite[Theorem~3]{KPSIGMA} we found the following
alternative representation for N{\o}rlund's coefficients:
\begin{gather}\label{eq:Norlund-Bernoulli}
g_n^{p}(\x;\y)=\sum\limits_{r=0}^n\frac{(1+\alpha-\nu(\x;\y)-n)_{n-r}}{(n-r)!}l_r(\x;\y;\alpha)\Be^{(n+\nu(\x;\y)-\alpha)}_{n-r}(1-\alpha),
\end{gather}
where $\nu(\x;\y)$ is defined in \eqref{eq:notation}, while the parameter $\alpha$ can be chosen arbitrarily. The coefficients $l_r(\x;\y;\alpha)$ are found from the recurrence relation
\begin{gather}\label{eq:lr}
l_r(\x;\y;\alpha)=\frac{1}{r}\sum\limits_{m=1}^r q_m(\x;\y;\alpha) l_{r-m}(\x;\y;\alpha)
\end{gather}
with $l_0=1$ and
\begin{gather}\label{eq:qm}
q_m(\x;\y;\alpha)=\frac{(-1)^{m+1}}{m+1}\Bigg(\Be_{m+1}(\alpha)+\sum_{j=1}^{p-1}\Be_{m+1}(x_j)
-\sum_{j=1}^{p}\Be_{m+1}(y_j)\Bigg)\!.
\end{gather}
The recurrence \eqref{eq:lr} can be solved giving the following explicit expression for $l_{r}(\x;\y;\alpha)$ \cite[formula~(4.1)]{QSL} (we use the abbreviated notation $q_m:=q_m(\x;\y;\alpha)$):
\begin{gather*}
l_{r}(\x;\y;\alpha)=\sum\limits_{\substack{k_1+2k_2+\cdots+rk_r=r,\\k_i\ge0}} \frac{q_1^{k_1}(q_2/2)^{k_2}\cdots (q_r/r)^{k_r}}{k_1!k_2!\cdots k_r!}
=\sum\limits_{n=1}^{r}\frac{1}{n!}
\sum\limits_{k_1+k_2+\cdots+k_n=r}\prod\limits_{i=1}^{n}\frac{q_{k_i}}{k_i}.
\end{gather*}
Another way to write this formula is to invoke the complete exponential Bell polynomials $Y_r$ generated by \cite[formula~(11.9)]{Charalambides}
\begin{gather*}
\exp\Bigg(\sum_{m=1}^{\infty}z_m\frac{t^m}{m!}\Bigg)=1+\sum_{r=1}^{\infty}Y_r(z_1,\dots,z_r)\frac{t^r}{r!}
\end{gather*}
and written explicitly as \cite[formula~(11.1)]{Charalambides}
\begin{gather*}
Y_r(z_1,\dots,z_r)=\sum\limits_{\substack{k_1+2k_2+\cdots+rk_r=r,\\k_i\ge0}} \frac{r!}{k_1!k_2!\cdots k_r!}(z_1/1!)^{k_1}(z_2/2!)^{k_2}\cdots(z_r/r!)^{k_r}.
\end{gather*}
These polynomials can also be found using a determinantal expression, see \cite[p.~203]{Collins}. Comparing the above explicit formulas for $l_{r}(\a;\b;\alpha)$ and $Y_r$ we conclude that
\begin{gather*}
l_{r}(\x;\y;\alpha)=\frac{1}{r!}Y_r(\hat{q}_1,\hat{q}_2,\dots,\hat{q}_r),
\end{gather*}
where $\hat{q}_m=(m-1)!q_m$ with $q_m$ from \eqref{eq:qm}. Substituting this into \eqref{eq:Norlund-Bernoulli} we obtain the second main result of this section.
\begin{Theorem}\label{th:NorlundBell}
	For arbitrary $\alpha$ the following identity is true:
	\begin{gather}\label{eq:gnBell}
g_n^{p}(\x;\y)=\sum\limits_{r=0}^n\frac{(1+\alpha-\nu(\x;\y)-n)_{n-r}}{r!(n-r)!}Y_r(\hat{q}_1,\hat{q}_2,\dots,\hat{q}_r)\Be^{(n+\nu(\x;\y)-\alpha)}_{n-r}(1-\alpha),
	\end{gather}
	where $\hat{q}_m=(m-1)!q_m$ with $q_m$ defined in \eqref{eq:qm}. In particular, the multiple hypergeometric series \eqref{eq:Norlund-explicit1} and \eqref{eq:NorlundMultiple} are equal to the right-hand side of~\eqref{eq:gnBell}.
\end{Theorem}

\begin{Remark} Formula \eqref{eq:gnBell} gives single sum expression for the coefficients $g_n^{p}(\x;\y)$ manifestly symmetric in the components of $\x$ and $\y$. The computational complexity of this formula does not increase with growing $p$ unlike \eqref{eq:Norlund-explicit1} and \eqref{eq:NorlundMultiple}. Particular cases obtained by expressing~$g_n^{p}$ using \eqref{eq:gnp3}, \eqref{eq:gnp3N}, \eqref{eq:KampedeFeriet1}, \eqref{eq:KampedeFeriet2}, \eqref{eq:SrivF3-1} or \eqref{eq:SrivF3-2} lead to rather exotic identities connecting terminating~${}_3F_{2}$, Kamp\'e de F\'eriet and Srivastava $F^{(3)}$ functions with the complete exponential Bell and the Bernoulli--N{\o}rlund polynomials.
\end{Remark}

\begin{Example}%\label{example12}
Application of Theorem~\ref{th:NorlundBell} to \eqref{eq:gnp3N} by renaming variables and taking $\alpha=0$ yields
\begin{gather*}
{}_3F_2\left( \begin{matrix}-n,\sigma,\beta\\\delta,\eta\end{matrix} \right)
=\sum\limits_{r=0}^n\binom{n}{r}\frac{(1+\beta-\delta-\eta-n)_{n-r}}{(\delta)_n(\eta)_n}Y_r(\hat{q}_1,\hat{q}_2,\dots,\hat{q}_r)\Be^{(n+\delta+\eta-\beta)}_{n-r}(1),
\end{gather*}
where
\begin{gather*}
\hat{q}_m=\frac{(-1)^{m+1}(m-1)!}{m+1}
\\ \hphantom{\hat{q}_m=}
{}\times\big\{2\Be_{m+1}(0)+\Be_{m+1}(\sigma-\beta)-\Be_{m+1}(\sigma)-\Be_{m+1}(\delta-\beta)-\Be_{m+1}(\eta-\beta)\big\}.
\end{gather*}
\end{Example}

\begin{Example}%\label{example13}
By application of Theorem~\ref{th:NorlundBell} with $\alpha=0$ to \eqref{eq:KampedeFeriet1} and renaming variables we have
\begin{gather*}
F^{2:2:1}_{2:1:0}\!\left(\left. \begin{matrix}
-n,\delta:\sigma,\beta:\gamma\\1-\mu-n,\eta:\delta:-\end{matrix}\right|1,1\!\right)
\\ \qquad
{}=\frac{1}{(\mu)_{n}(\eta)_{n}}
\sum\limits_{r=0}^n(-1)^{n-r}\binom{n}{r}(\sigma+\beta+\mu+\gamma-r)_{n-r}Y_r(\hat{q}_1,\hat{q}_2,
\dots,\hat{q}_r)\Be^{(n+\sigma+\beta+\mu+\gamma)}_{n-r}(1)
\end{gather*}
with
\begin{gather*}
\hat{q}_m=\frac{(-1)^{m+1}(m-1)!}{m+1}\Bigg\{\Be_{m+1}(0)+\sum_{j=1}^{3}\Be_{m+1}(x_j)
-\sum_{j=1}^{4}\Be_{m+1}(y_j)\Bigg\},
\end{gather*}
where
\begin{gather*}
x_1=0,\qquad x_2=\sigma+\beta-\delta,\qquad x_3=\sigma+\beta-\eta+\gamma,
\\
y_1=\sigma,\qquad y_2=\beta,\qquad y_3=\sigma+\beta-\delta+\gamma,\qquad y_4=\sigma+\beta-\eta+\gamma+\mu.
\end{gather*}
If we use \eqref{eq:KampedeFeriet2} instead of \eqref{eq:KampedeFeriet1} we get a similar expression for $F^{3:2:0}_{2:2:0}$.
\end{Example}

\begin{Example}\label{example14} By application of Theorem~\ref{th:NorlundBell} with $\alpha=0$ to \eqref{eq:SrivF3-1} and renaming variables we have
\begin{gather*}
F^{(3)}\!\left[
\left. \begin{matrix}-n,\eta::\delta;-;-:\sigma,\beta;\gamma;\mu
\\1-\zeta-n,\xi::\eta;-;-:\delta;-;-\end{matrix}\right|1,1,1\right]
\\ \qquad
{}=\sum\limits_{r=0}^n\binom{n}{r}\frac{(1-\sigma-\beta-\gamma-\mu-\zeta-n)_{n-r}}{(\xi)_{n}(\zeta)_{n}}
Y_r(\hat{q}_1,\hat{q}_2,\dots,\hat{q}_r)\Be^{(n+\sigma+\beta+\gamma+\mu+\zeta)}_{n-r}(1)
\end{gather*}
with
\begin{gather*}
\hat{q}_m=\frac{(-1)^{m+1}(m-1)!}{m+1}\Bigg\{\Be_{m+1}(0)+\sum_{j=1}^{4}\Be_{m+1}(x_j)
-\sum_{j=1}^{5}\Be_{m+1}(y_j)\Bigg\},
\end{gather*}
where
\begin{gather*}
x_1=0,\quad\ x_2=\sigma+\beta-\delta,\quad\ x_3=\sigma+\beta-\eta+\gamma,\quad\ x_4=\beta+\gamma+\mu+\sigma-\xi,\quad\ y_1=\sigma,
\\
y_2=\beta,\quad\ y_3=\sigma+\beta-\delta+\gamma,\quad\ y_4=\sigma+\beta-\eta+\gamma+\mu,\quad\
y_5=\beta+\gamma+\mu+\sigma-\xi+\zeta.
\end{gather*}
Combining with Example~\ref{example11} we can write a similar expression for another version of $F^{(3)}$.
\end{Example}

We conclude this section with the remark that explicit expressions for $g_{1}^{p}$ and $g_{2}^{p}$ for any $p$ can be found in \cite[Theorem~3.1]{KPSIGMA}, while $p$-th order recurrence relation with respect to $n$ was found by N{\o}rlund \cite[formula~(1.28)]{Norlund}, see also \cite[equation~(2.7)]{KPSIGMA}.

\section[Function p+1 Fp at and near the singular unity revisited]
{Function $\boldsymbol{{}_{p+1}F_p}$ at and near the singular unity revisited}\label{section3}

The following expansion for ${}_{p+1}F_{p}(z)$ around the singular point $z=1$ was first obtained by N{\o}rlund \cite{Norlund} with further contributions by Olsson \cite{Olsson}, B\"{u}hring \cite{Buhring}, Saigo and Srivastava \cite{SaigoSrivastava}:
\begin{gather}\label{eq:KPSIGMA2.20}
\frac{\Gamma(\a)}{\Gamma(\b)}{}_{p+1}F_p\left.\! \left(\begin{matrix}\a\\\b\end{matrix}\right\vert z\right)
=(1-z)^{\nu}\sum\limits_{k=0}^{\infty}f_p(k;\a;\b)(1-z)^k+\sum\limits_{k=0}^{\infty}h_p(k;\a;\b)(1-z)^k,
\end{gather}
where $\nu=\nu(\a;\b)$ is assumed to be non-integer. Various expressions are known for the coefficients $f_p(k;\a;\b)$ and $h_p(k;\a;\b)$, see \cite[formula~(2.20)]{KPSIGMA} for details. We first note that this expansion can be derived in a rather straightforward manner from representation \eqref{eq:Frepr}. Indeed, substituting expansion \eqref{eq:Norlund} into \eqref{eq:Frepr} and integrating term by term, in view of Euler's integral representation \cite[Theorem~2.2.4]{AAR} for ${}_2F_1{}$, we obtain (after renaming $(\sigma,\a)\to\a$)
\begin{gather*}
\frac{\Gamma\big(\a_{[1,2]}\big)}{\Gamma(\b)}\, {}_{p+1}F_p\!\left.\left(\begin{matrix}\a\\\b\end{matrix}\right\vert z\right)
=\sum\limits_{n=0}^{\infty}\frac{g_n^{p}\big(\a_{[1,2]};\b\big)}{\Gamma\bigl(\nu_{[1,2]}+n\bigr)}\,
{}_{2}F_{1}\!\left. \left(\begin{matrix}a_1,a_2\\\nu_{[1,2]}+n\end{matrix}\right\vert z\right)\!,
\end{gather*}
where $\nu_{[1,2]}=b_1+\cdots+b_p-a_3-\cdots-a_p$, and $g_n^{p}(\cdot)$ are N{\o}rlund's coefficients defined in \eqref{eq:Norlund}. Next, applying formula \cite[formula~(2.3.13)]{AAR} connecting ${}_{2}F_{1}(z)$ and ${}_{2}F_{1}(1-z)$ in the case of non-integer parametric excess (equal to $\nu+n$ here) and rearranging slightly we arrive at
\begin{gather*}
\frac{\sin(\pi\nu)\Gamma(\a)}{\pi\Gamma(\b)}\, {}_{p+1}F_p\!\!\left. \left(\begin{matrix}\a\\\b\end{matrix}\right\vert z\right)
\!=\!\!\sum\limits_{k=0}^{\infty}\!\frac{\Gamma(a_1+k)\Gamma(a_2+k)}{\Gamma(1-\nu+k)k!}(1-z)^k
\sum\limits_{n=0}^{\infty}
\!\frac{g_n^{p}\big(\a_{[1,2]};\b\big)(\nu-k)_n}{\Gamma\bigl(\nu\!+a_1\!+n\bigr)\Gamma\bigl(\nu+a_2+n\bigr)}
\\ \qquad
{}-(1-z)^{\nu}
\sum\limits_{k=0}^{\infty}\frac{(\nu+a_1)_k(\nu+a_2)_k}{\Gamma(1+\nu+k)}(1-z)^k\sum\limits_{n=0}^{k}
\frac{(-1)^{n}g_n^{p}\big(\a_{[1,2]};\b\big)}{(\nu+a_1)_n(\nu+a_2)_n(k-n)!}.
\end{gather*}
Using the recurrence relation \eqref{eq:NorlCoefidentity} we obtain the formulas for the coefficients $f_p(k;\a;\b)$ and $h_p(k;\a;\b)$ from \eqref{eq:KPSIGMA2.20} in terms of N{\o}rlund's coefficients:
\begin{subequations}
	\begin{gather}\label{eq:f-g}
	f_p(k;\a;\b)=\frac{\Gamma(-\nu)}{(1+\nu)_{k}}g_k^{p+1}(1-\b;1-\a)\qquad
\big(\text{in particular}~f_p(0;\a;\b)=\Gamma(-\nu)\big)
	\end{gather}
	and
	\begin{gather}%\label{eq:h-g}
	h_p(k;\a;\b)=\frac{\Gamma(\nu)\Gamma(a_1+k)\Gamma(a_2+k)}{(1-\nu)_{k}k!}
	\sum\limits_{n=0}^{\infty}
	\frac{g_n^{p}\big(\a_{[1,2]};\b\big)(\nu-k)_n}{\Gamma\bigl(\nu+a_1+n\bigr)\Gamma\bigl(\nu+a_2+n\bigr)}.
	\end{gather}
\end{subequations}
Taking $p=2$ and substituting formulas \eqref{eq:Norlund-p2} and \eqref{eq:gnp3} into the above expressions, expansion~\eqref{eq:KPSIGMA2.20} recovers \cite[Theorem~1]{Buhring87} or a slightly different version of it if we use \eqref{eq:gnp3N} instead of~\eqref{eq:gnp3}. If, on the other hand, we take $p=3$ and employ \eqref{eq:gnp3} or \eqref{eq:gnp3N} in the series for $h_p(k;\a;\b)$ and~\eqref{eq:KampedeFeriet1} or \eqref{eq:KampedeFeriet2} in the formula for $f_p(k;\a;\b)$ we arrive at several expansions for~${}_4F_{3}(z)$ which may be new. Similarly, for $p=4$ we arrive at expansions for~${}_5F_{4}(z)$ in terms of Srivastava's $F^{(3)}$ once we employ \eqref{eq:SrivF3-1} or \eqref{eq:SrivF3-2}. Generalizations of these ideas to the Fox--Wright function were explored in \cite{KPJMAA2020}.

Let us also remark that a similar substitution of expansion \eqref{eq:Norlund} into the Laplace transform representation
\begin{gather*}
{\rm e}^{-z}{}_{p}F_p\left.\! \left(\begin{matrix}\a\\\b\end{matrix}\right\vert z\right)=\frac{\Gamma(\b)}{\Gamma(\a)}\int_0^{1}{\rm e}^{-zu}G^{p,0}_{p,p}\!\left(1-u\left|\begin{matrix} \b-1 \\ \a-1 \end{matrix}\right.\right){\rm d}u
\end{gather*}
obtained from \cite[equation~(4)]{KLJAT2017} or \cite[formula~(1.2)]{KPITSF2017} via multiplication by ${\rm e}^{-z}$ and change of variable $u=1-t$, leads immediately to asymptotic expansion of
${}_pF_p(z)$ as $z\to\infty$ recovering the main result of~\cite{VolkmerWood}. Further development of these ideas can be found in~\cite{LPKITSF2020}.

Next, we note that for $\operatorname{Re}(\nu(\a;\b))>0$ by setting $z=1$ in expansion \eqref{eq:KPSIGMA2.20}, we get
\begin{gather}\label{eq:hp0}
h_p(0;\a;\b)=\frac{\Gamma(\a)}{\Gamma(\b)}\, {}_{p+1}F_p\!\left. \left(\begin{matrix}\a\\\b\end{matrix}\right\vert 1\right)\!.
\end{gather}
Hence, omitting the unit argument from the notation of ${}_{p+1}F_p$, we arrive at
\begin{gather}\label{eq:p+1Fp(1)}
{}_{p+1}F_p\!\left(\begin{matrix}\a\\\b\end{matrix}\right)
=\frac{\Gamma(\b)}{\Gamma\big(\a_{[1,2]}\big)}\sum\limits_{n=0}^{\infty}
\frac{\Gamma(\nu+n)g_n\big(\a_{[1,2]};\b\big)}{\Gamma(\nu+a_1+n)\Gamma(\nu+a_2+n)}.
\end{gather}
It follows from the estimate in \cite[Lemma~1.1]{LPKITSF2020} that the series on the right-hand side converges if $\operatorname{Re}\big(\a_{[1,2]}\big)>0$, while the series on the left hand side converges if $\operatorname{Re}(\nu(\a;\b))>0$, so that the expression on the right gives the analytic continuation in parameters for ${}_{p+1}F_p(\a;\b;1)$. Formula~\eqref{eq:p+1Fp(1)} in a slightly different form was obtained by N{\o}rlund \cite[formula~(4.6)]{Norlund}. It is also equivalent to \cite[Theorem~1]{Buhring}, modulo simple transformation of coefficients in B\"{u}hring's formula. We further remark that although for $\operatorname{Re}(\nu(\a;\b)))<0$ the hypergeometric series on~the left hand side of \eqref{eq:p+1Fp(1)} diverges, in view of representation \eqref{eq:Gppexpansion1} and expansion \eqref{eq:NorlundSpecial} for~$-\nu(\a;\b)=m\in\N_0$, we obtain the limit formula
\begin{gather*}
\lim\limits_{x\to1-}\sum\limits_{k=1}^{p}
\frac{\Gamma\big(\b_{[k]}-b_{k}\big)}{\Gamma(\cb-b_{k})}x^{b_k}\, {}_{p}
F_{p-1}\! \left( \begin{matrix}1-\cb+b_k\\1-\b_{[k]}+b_k\end{matrix} \,\vline\,x\!\right)
=g_{m+1}\big(\b_{[\omega]};\cb\big),
\end{gather*}
where $\omega\in\{1,\dots,p\}$ is arbitrary as we mentioned below \eqref{eq:NorlundSpecial}.

In \cite[Theorem~3.10]{KPSIGMA} by reinterpreting some N{\o}rlund's results, we gave the following expression for the coefficients $h_p(m;\a;\b)$ (after some change of notation):
\begin{gather}\label{eq:G2ppp}
h_p(m;\a;\b)=\frac{1}{\pi\sin(\pi\nu)}\sum_{k=1}^{p} \frac{\sin(\pi(b_k-\a))}{\sin\big(\pi\big(b_k-\b_{[k]}\big)\big)}D_m^{[k]},
\end{gather}
where, according to \cite[formulas~(2.17), (2.18)]{KPSIGMA},
\begin{gather}
\notag
D_m^{[k]}=\frac{\Gamma(1+\a-b_k)\Gamma(a_1+m)\Gamma(a_2+m)}
{\Gamma\big(1+\b_{[k]}-b_k\big)\Gamma(1+a_1+a_2-b_k+m)m!}
\\
\hphantom{D_m^{[k]}=}{}
{}\times\sum_{j=0}^{\infty}\frac{(1+a_1-b_k)_j(1+a_2-b_k)_j}{j!(1+a_1+a_2-b_k+m)_j}\,
{}_{p}F_{p-1}\left(\begin{matrix}-j,1+\a_{[1,2]}-b_k\\1+\b_{[k]}-b_k\end{matrix} \right)\label{eq:Norlund5.35}
\\
\hphantom{D_m^{[k]}}{}
{}=\frac{\Gamma(1+\a-b_k)\Gamma(b_k+m)}{\Gamma\big(1+\b_{[k]}-b_k\big)m!}
\sum_{j=0}^{\infty}\frac{(1+a_1-b_k)_j}{(a_1+m)_{j+1}}\,
{}_{p+1}F_{p} \left(\begin{matrix}-j,1+\a_{[1]}-b_k\\ 1,1+\b_{[k]}-b_k\end{matrix} \right)\!.\label{eq:Norlund5.36}
\end{gather}
Combining \eqref{eq:hp0} with \eqref{eq:G2ppp} and \eqref{eq:Norlund5.35}, we arrive at
\begin{Theorem}\label{th:p+1Fp(1)-new}
	The following representation holds true
\begin{gather} \frac{\Gamma\big(\a_{[1,2]}\big)}{\Gamma(\b)}\,{}_{p+1}F_p\!\left(\begin{matrix}\a\\\b\end{matrix}\right)
=\frac{\pi}{\sin(\pi\nu)}\!\sum\limits_{k=1}^{p}\!\frac{\Gamma\big(b_k-\b_{[k]}\big)}{\Gamma(b_k-\a)}\nonumber
\\ \hphantom{\frac{\Gamma\big(\a_{[1,2]}\big)}{\Gamma(\b)}\, {}_{p+1}F_p\!
\left(\begin{matrix}\a\\\b\end{matrix}\right)=}
{}\times\sum\limits_{n=0}^{\infty}\!\frac{(1-b_k+a_1)_n(1-b_k+a_2)_n}{\Gamma(1-b_k+a_1+a_2+n)n!}
{}_{p}F_{p-1}\!\left(\begin{matrix}-n,1-b_k+\a_{[1,2]}\\1-b_k+\b_{[k]}\end{matrix}\right)\!,
	\label{eq:p+1Fp(1)-new}
\end{gather}
	where $\nu=\nu(\a;\b)$ is defined in \eqref{eq:notation} and the series on the right-hand side converges for $\operatorname{Re}\big(\a_{[1,2]}\big)>0$.
\end{Theorem}

\begin{Remark} If we use \eqref{eq:Norlund5.36} instead of \eqref{eq:Norlund5.35} we obtain
\begin{gather*}
\frac{\Gamma(\a)}{\Gamma(\b)}\, {}_{p+1}F_p\!\left(\begin{matrix}\a\\\b\end{matrix}\right)
=\frac{\pi}{\sin(\pi\nu)}\sum\limits_{k=1}^{p}\!\frac{\Gamma(b_k)\Gamma\big(b_k-\b_{[k]}\big)}{\Gamma(b_k-\a)}
\\ \hphantom{\frac{\Gamma(\a)}{\Gamma(\b)}\, {}_{p+1}F_p\!\left(\begin{matrix}\a\\\b\end{matrix}\right)=}
{}\times\sum_{n=0}^{\infty}\frac{(1-b_k+a_1)_n}{(a_1)_{n+1}}\,
{}_{p+1}F_{p} \left( \begin{matrix}-n,1-b_k+\a_{[1]}\\ 1,1-b_k+\b_{[k]}\end{matrix} \right)\!.
\end{gather*}
This series also converges for $\operatorname{Re}\big(\a_{[1,2]}\big)>0$.
However, this expansion seems to be less useful than~\eqref{eq:p+1Fp(1)-new} as the higher order hypergeometric function appears on the right-hand side.
\end{Remark}

\begin{Remark} Note that formula \eqref{eq:p+1Fp(1)-new} is fundamentally different from both \eqref{eq:p+1Fp(1)} and \cite[Theorem~1]{Buhring}, since for any value of $p$ we have the triple summation on the right-hand side (viewing the hypergeometric polynomial as a single sum), while \eqref{eq:p+1Fp(1)} and \cite[Theorem~1]{Buhring} represent $(p-1)$-fold summations.
If condition $\operatorname{Re}\big(\a_{[1,2]}\big)>0$ is violated, then we can use the following straightforward decomposition \cite[inequality~(31)]{KLJAT2017}
\begin{gather*}\label{eq:pFq-decompose}
{}_{p+1}F_p\!\left(\begin{matrix}\a\\\b\end{matrix}\right)=
\sum_{k=0}^{M-1}\frac{(\a)_k}{(\b)_kk!}
+\frac{(\a)_Mz^{M}}{(\b)_MM!}\, {}_{p+2}F_{p+1}\!\left(\begin{matrix}\a+M,1\\\b+M,M+1\end{matrix}\right)
\qquad\forall M\in\N_0.
\end{gather*}
Choosing $M$ sufficiently large we can always ensure the condition $\operatorname{Re}(\a_{[1,2]}+M)>0$ in the second term, while the first term is a rational function. Note, that the parametric excess in the second term on the right-hand side
\begin{gather*}
\nu((\a+M,1);(\b+M,M+1))=\nu(\a;\b),
\end{gather*}
which is the parametric excess of the function on the left hand side irrespective of $M$. Hence, for $\operatorname{Re}(\nu(\a;\b))<0$ the above decomposition alone does not help to construct the analytic continuation. However, when combined with Theorem~\ref{th:p+1Fp(1)-new} or with formula \eqref{eq:p+1Fp(1)} this decomposition furnishes the analytic continuation to all finite values of parameters save for poles at $-\nu(\a;\b)\in\N_{0}$ and $-b_j\in\N_0$.
\end{Remark}

Next, we explore some consequences of \eqref{eq:p+1Fp(1)-new}. Taking $p=2$ and applying the Chu--Van\-der\-mon\-de identity to the ${}_2F_{1}$ polynomial on the right, after some simplifications we arrive at
\begin{gather*}
\frac{\sin(\pi\nu)\Gamma(a_3)}{\pi\Gamma(\b)}\, {}_{3}F_{2}\!\left(\begin{matrix}\a\\\b\end{matrix}\right)
\\ \qquad
{}=\frac{\Gamma(b_1-b_2)}{\Gamma(b_1-\a)\Gamma(1-b_1+a_1+a_2)}{}_{3}F_{2}\!
\left(\begin{matrix}1-b_1+\a_{[3]},b_2-a_3\\1-b_1+b_2,1-b_1+a_1+a_2\end{matrix}\right)
+\mathrm{idem}(b_1;b_2),
\end{gather*}
where the symbol $\mathrm{idem}(b_1;b_2)$ after an expression means that the preceding expression is repeated with $b_1$, $b_2$ interchanged. This formula is yet another instance of the three-term Thomae relations. It can be obtained by eliminating $F_n(0)$ from the pair of equations \cite[formulas~3.7(4) and~3.7(6)]{Bailey}.

Before we move forward we need the following lemma.
\begin{Lemma}%\label{lm:pFqfactor}
Suppose $n\in\N_0$, $m\in\N$. Then
\begin{gather}\label{eq:factor2}
{}_{p}F_q\! \left(\begin{matrix}-n,-m,\a\\\b,1-\alpha-n\end{matrix}\right)
=\frac{(\alpha-m)_{n}}{(\alpha)_{n}}\frac{(1-\lambda_1)_n\cdots(1-\lambda_m)_n}{(-\lambda_1)_n
\cdots(-\lambda_m)_n},
\end{gather}
where $\lambda_1,\dots,\lambda_m$ are the zeros of the polynomial
\begin{gather}\label{eq:Pm}
P_m(x)=\sum\limits_{k=0}^{m}\frac{(-m)_k(\a)_k(-x)_k(x+\alpha-m)_{m-k}}{(-1)^k(\b)_kk!}.
\end{gather}
In particular, for $m=1$, $\lambda_1=(1-\alpha)(\b)_{1}/[(\b)_{1}-(\a)_{1}]$.
\end{Lemma}
\begin{proof}
In order to establish (\ref{eq:factor2}) first note that for integer $0\le{k}\le{m}$:
\begin{gather*}
\frac{(\alpha)_n}{(1-\alpha-n)_k}=(-1)^k(\alpha)_{n-m}(\alpha+n-m)_{m-k}.
\end{gather*}
This formula is also true if $m>n$ with the standard convention $(a)_{r}=(-1)^{r}/(1-a)_{-r}$ for integer $r<0$. Applying this relation we get
\begin{gather*} (\alpha)_{n}\cdot{}_{p}F_q\! \left(\begin{matrix}-n,-m,\a\\\b,1-\alpha-n\end{matrix}\right)
=(\alpha)_{n-m}P_m(n),
\end{gather*}
where $P_m(x)$ is defined by (\ref{eq:Pm}). Factoring $P_m(x)$ we get:
\begin{gather*}
P_m(n)=A(n-\lambda_1)(n-\lambda_2)\cdots(n-\lambda_m)=A(-\lambda_1)
\cdots(-\lambda_m)\frac{(1-\lambda_1)_n\cdots(1-\lambda_m)_n}{(-\lambda_1)_n\cdots(-\lambda_m)_n},
\end{gather*}
where $A$ denotes the coefficient at $x^m$ in $P_m(x)$. It remains to note that
\begin{gather*}
A(-\lambda_1)(-\lambda_2)\cdots(-\lambda_m)=P_m(0)=(\alpha-m)_m
\end{gather*}
and
\begin{gather*}
(\alpha-m)_m(\alpha)_{n-m}=(\alpha-m)_{n}.\!\!
\tag*{\qed}
\end{gather*}\renewcommand{\qed}{}
\end{proof}

We now take $p=3$ and $a_3=b_3+m$, $m\in\N$, in (\ref{eq:p+1Fp(1)-new}). The term corresponding to $k=3$ vanishes and we are left with two infinite series with terms involving
\begin{gather*}
{}_{3}F_{2}\!\left(\begin{matrix}-n,1-b_1+\a_{[1,2]}\\1-b_1+\b_{[1]}\end{matrix}\right)\qquad
\text{and}\qquad {}_{3}F_{2}\!\left(\begin{matrix}-n,1-b_2+\a_{[1,2]}\\1-b_2+\b_{[2]}\end{matrix}\right)\!.
\end{gather*}
If we apply Sheppard's transformation \eqref{eq:Sheppard} to each of these two functions
we obtain (keeping in mind that $a_3=b_3+m$, $\nu=\nu(\a;\b)$):
\begin{gather*}
\frac{\Gamma\big(\a_{[1,2]}\big)\sin(\pi\nu)}{\Gamma(\b)\pi}\,{}_{4}F_{3}\!
\left(\begin{matrix}\a\\\b\end{matrix}\right)
\\[.5ex] \qquad
{}=\frac{\Gamma\big(b_1-\b_{[1]}\big)}{\Gamma(b_1-\a)}\!
\sum\limits_{n=0}^{\infty}\!\frac{(1-b_1+a_1)_n(1-b_1+a_2)_n(b_{2}-a_4)_n}
{\Gamma(1-b_1+a_1+a_2+n)(1-b_1+b_{2})_nn!}
\\[.5ex] \qquad\phantom{=}
{}\times{}_{3}F_{2}\!\left(\begin{matrix}-n,-m,1-b_1+a_4\\1-b_1+b_3,1-b_{2}+a_4-n\end{matrix}\right)
+\mathrm{idem}(b_1;b_2)
=\frac{\Gamma\big(b_1-\b_{[1]}\big)}{\Gamma(b_1-\a)}\!
\\[.5ex] \qquad\phantom{=}
{}\times\sum\limits_{n=0}^{\infty}\!\frac{(1\!-b_1\!+a_1)_n(1\!-b_1\!+a_2)_n(b_2\!-a_4\!-m)_{n}
(1\!-\lambda_1)_n
\cdots(1\!-\lambda_m)_n}{\Gamma(1-b_1+a_1+a_2+n)(1-b_1+b_{2})_n(-\lambda_1)_n\cdots(-\lambda_m)_nn!}
+\mathrm{idem}(b_1;b_2)
\\[.5ex] \qquad
{}=\frac{\Gamma\big(b_1-\b_{[1]}\big)}{\Gamma(b_1\!-\a)\Gamma(1-b_1\!+a_1\!+a_2)}
{}_{m+3}F_{m+2}\!\left(\begin{matrix}1\!-b_1\!+a_1,1\!-b_1\!+a_2,b_2\!-a_4\!-m,1\!-\boldsymbol{\lambda}
\\1-b_1+a_1+a_2, 1-b_1+b_{2},-\boldsymbol{\lambda}\end{matrix}\right)
\\[.5ex] \qquad\phantom{=}
{}+\mathrm{idem}(b_1;b_2),
\end{gather*}
where the vector $\boldsymbol{\lambda}=(\lambda_1,\dots,\lambda_m)$ comprises the zeros of the polynomial
\begin{gather*}
P_m(a_4;b_1,b_2,b_3; x) =\sum\limits_{k=0}^{m}\frac{(-m)_k(-x)_k(x+b_2-a_4-m)_{m-k}(1-b_1+a_4)_k}{(-1)^k(1-b_1+b_3)_kk!}
\end{gather*}
and $\mathrm{idem}(b_1;b_2)$ contains the zeros of the polynomial $P_m(a_4;b_2,b_1,b_3; x)$. The simplest and the most interesting particular case is $m=1$. In this case we get the following three-term relation for ${}_4F_{3}(1)$ with one unit shift in the parameters:
\begin{gather*}
\frac{b_3\Gamma(a_4)\sin(\pi\nu)}{\Gamma(b_1)\Gamma(b_2)\pi}\,{}_{4}F_{3}\!
\left(\begin{matrix}\a\\\b\end{matrix}\right)
\\ \qquad
{}=\frac{\Gamma\big(b_1-\b_{[1]}\big)}{\Gamma(b_1-\a)\Gamma(1-b_1+a_1+a_2)}\,
{}_{4}F_{3}\!\left(\begin{matrix}1-b_1+a_1,1-b_1+a_2,b_2-a_4-1,1-\lambda_1\\1-b_1+a_1+a_2, 1-b_1+b_{2},-\lambda_1\end{matrix}\right)
\\ \qquad\phantom{=}
{}+\frac{\Gamma\big(b_2-\b_{[2]}\big)}{\Gamma(b_2-\a)\Gamma(1-b_2+a_1+a_2)}\,
{}_{4}F_{3}\!\left(\begin{matrix}1-b_2+a_1,1-b_2+a_2,b_1-a_4-1,1-\lambda_2\\1-b_2+a_1+a_2, 1-b_2+b_{1},-\lambda_2\end{matrix}\right)\!,
\end{gather*}
where $a_3=b_3+1$, $\nu=b_1+b_2-a_1-a_2-a_4-1$ and
\begin{gather*}
\lambda_1=\frac{(1-b_2+a_4)(1-b_1+b_3)}{b_3-a_4},\qquad
\lambda_2=\frac{(1-b_1+a_4)(1-b_2+b_3)}{b_3-a_4}.
\end{gather*}

We now turn our attention to the consequences of \eqref{eq:p+1Fp(1)}.
For $p=2$ substitution $g_n(a;\b)=(b_1-a)_{n}(b_2-a)_{n}/n!$ (see \eqref{eq:Norlund-p2}) gives yet another proof of the two-term Thomae relation \cite[Corollary~3.3.6]{AAR}.
For $p=3$ substituting \eqref{eq:gnp3N} into (\ref{eq:p+1Fp(1)}) yields:
\begin{gather}
{}_{4}F_3\!\left(\begin{matrix}\a\\\b\end{matrix}\right)
=\frac{\Gamma(\b)\Gamma(\nu)}{\Gamma(\a_{[1,2]}+\nu)\Gamma(a_1)\Gamma(a_2)}\nonumber
\\ \hphantom{{}_{4}F_3\!\left(\begin{matrix}\a\\\b\end{matrix}\right)=}
{}\times\sum\limits_{n=0}^{\infty}\frac{(\nu)_n(\nu-\b_{[1]}+a_3+a_4)_n}{(\a_{[1,2]}+\nu)_nn!}\,
{}_{3}F_2\!\left(\begin{matrix}-n,b_1-a_1,b_1-a_2\\\nu-\b_{[1]}+a_3+a_4\end{matrix}\right)\!.
\label{eq:4F3sum3F2}
\end{gather}
This formula differs from \cite[formulas~(2.14) and~(2.10)]{Buhring}, but can be reduced to it by an application of Whipple's transformation for terminating ${}_{3}F_2(1)$ \cite[p.~142, top]{AAR}. If we put $a_1=b_1+1$ in \eqref{eq:4F3sum3F2} we arrive at a two-term transformation for ${}_4F_{3}$ with one unit shift studied by two of us recently in \cite[identity~(7)]{KP2020}. Further consequences of \eqref{eq:p+1Fp(1)} are explored in the following section.

\section[Transformations of 5F4 with two unit shifts]
{Transformations of $\boldsymbol{{}_{5}F_4}$ with two unit shifts}\label{section4}

In our recent paper \cite{KP2020} we have studied a group of transformations of ${}_{4}F_3(1)$ with one unit shift in the parameters. We have shown that this group is generated by two-term Thomae transformations and contiguous relations for ${}_{3}F_2(1)$. In this section we will demonstrate that a similar group can be generated for ${}_{5}F_4$ with two unit shifts. Let us note that the study of the summation and transformation formulas for the generalized hypergeometric function with integral parameter differences was initially motivated by problems from mathematical physics, see, for instance, \cite{Milgram, Minton, SS2015}. The most general linear transformations for this class of hypergeometric series at arbitrary argument were discovered by Miller and Paris~\cite{MP2013} (see an alternative derivation in \cite{KPChapter2019}), while quadratic and cubic transformations were found recently by Maier~\cite{Maier}.
\begin{Theorem}\label{th:5F4woterm}
The following two-term transformation holds\emph{:}
\begin{gather}
{}_{5}F_4\left.\! \left(\begin{matrix}a,b,c,f+1, h+1\\d,e,f,h\end{matrix}\right.\right)\nonumber
\\ \qquad
{}=\frac{\Gamma(d)\Gamma(e)\Gamma(s)}{f h\Gamma(c)\Gamma(s+a)\Gamma(s+b)}\,
{}_{5}F_4\left.\! \left(\begin{matrix}s, e-c-2, d-c-2,1-\gamma_1,1-\gamma_2\\s+a,s+b,-\gamma_1,-\gamma_2\end{matrix}\right.\right)\!,
\label{eq:5F4twoterm}
\end{gather}
where $s=d+e-a-b-c-2$ and $\gamma_1$, $\gamma_2$ are the roots of the second degree polynomial
\begin{gather*}
P_2(x)={}_{4}F_3\left.\! \left(\begin{matrix}-2, -x,h-c-1,\eta+1\\e-c-2,d-c-2,\eta\end{matrix}\right.\right)
\end{gather*}
with $\eta=-2(h-c-1)/(h-f-1)$.
\end{Theorem}
\begin{proof}
Formula \eqref{eq:p+1Fp(1)} applied to the left hand side of \eqref{eq:5F4twoterm} has the form
\begin{gather}\label{eq:p+1Fpdd1}
{}_{5}F_4\!\left(\begin{matrix}a,b,f+1,h+1,c\\f,h,d,e\end{matrix}\right) =\frac{\Gamma(d)\Gamma(e)}{\Gamma(c)fh}\sum\limits_{n=0}^{\infty}\frac{\Gamma(s+n)g_n^{4}(\{f+1,h+1,c\};\{f,h,d,e\})}{\Gamma(s+a+n)\Gamma(s+b+n)},
\end{gather}
where $s=d+e-a-b-c-2$. According to \eqref{eq:nor} we have
\begin{gather*}
g_n^{4}(\{f+1,h+1,c\};\{f,h,d,e\})
\\ \qquad
{}=\frac{(e-c-2)_n(d-c-2)_n}{n!}
\sum\limits_{k=0}^{n}\frac{(-n)_k(-2)_k(h-c-1)_k}{k!(e-c-2)_k(d-c-2)_k}
\bigg(1+\frac{k(h-f-1)}{(-2)(h-c-1)}\bigg)
\\ \qquad
{}=\frac{(e-c-2)_n(d-c-2)_n}{n!}
\sum\limits_{k=0}^{n}\frac{(-n)_k(-2)_k(h-c-1)_k}{k!(e-c-2)_k(d-c-2)_k}\frac{(\eta+1)_k}{(\eta)_k},
\end{gather*}
where $\eta=-2(h-c-1)/(h-f-1)$. Hence,
\begin{gather*}
g_n^{4}(\{f\!+1,h\!+1,c\};\{f,h,d,e\})=\frac{(e\!-c\!-2)_n(d\!-c\!-2)_n}{n!}\,{}_{4}F_{3}
\!\left( \begin{matrix}-n,-2,h-c-1,\eta+1\\e-c-2,d-c-2,\eta\end{matrix} \right)\!.
\end{gather*}
Noting that the quadratic polynomial
\begin{gather*}
P_2(x)={}_{4}F_3\left.\! \left(\begin{matrix}-2, -x,h-c-1,\eta+1\\e-c-2,d-c-2,\eta\end{matrix}\right.\right)
\end{gather*}
satisfies $P_2(0)=1$, we conclude that $P_2(x)=(x-\gamma_1)(x-\gamma_2)/(\gamma_1\gamma_2)$, where $\gamma_1$, $\gamma_2$ are the roots of $P_2(x)$. Hence, $P_2(n)=(1-\gamma_1)_n(1-\gamma_2)_n/((-\gamma_1)_n(-\gamma_2)_n)$ and
\begin{gather*}
g_n^{4}(\{f+1,h+1,c\};\{f,h,d,e\})=\frac{(e-c-2)_n(d-c-2)_n(1-\gamma_1)_n(1-\gamma_2)_n}
{(-\gamma_1)_n(-\gamma_2)_{n}n!}.
\end{gather*}
Substituting the above formula into \eqref{eq:p+1Fpdd1}, we arrive at \eqref{eq:5F4twoterm}.
\end{proof}

Another transformation of a similar flavour as \eqref{eq:5F4twoterm} has been recently found by us in \cite[formula (59)]{KPGFmethod}, namely,
\begin{gather}
{}_{5}F_{4}\left.\! \left( \begin{matrix}
a,b,c,f+1,h+1\\d,e,f,h\end{matrix}\right.\right)=
\frac{((e-c-1)h+(d-a-b-1)(h-c))\Gamma(e)\Gamma(s)}{h\Gamma(s+c+1)\Gamma(e-c)}\nonumber
\\ \hphantom{{}_{5}F_{4}\left.\! \left( \begin{matrix}
a,b,c,f+1,h+1\\d,e,f,h\end{matrix}\right.\right)=}
{}\times{}_{5}F_{4}\left.\! \left( \begin{matrix}
d-a-1,d-b-1,c,\xi+1,\zeta+1\\d,e+d-a-b-1,\xi,\zeta\end{matrix}\right.\right)\!,
\label{eq:5F4-2unitshifts}
\end{gather}
where $s=d+e-a-b-c-2$, $d-a-1\ne0$, $d-b-1\ne0$ and
\begin{gather*}
\xi=h+\frac{(d-a-b-1)(h-c)}{e-c-1},\qquad \zeta=\frac{(d-a-1)(d-b-1)f}{(d-a-b-1)f+ab}.
\end{gather*}
Transformations \eqref{eq:5F4twoterm} and \eqref{eq:5F4-2unitshifts} can be iterated and composed with each other. Together with the obvious invariance with respect to permutations of the upper and the lower parameters they generate a group, which can be shown to be isomorphic to the direct product of two-term Thomae transformations for ${}_{3}F_{2}$ with contiguous relations for ${}_{3}F_{2}$. This claim can be verified using the approach from \cite{KP2020}. We further believe that similar transformations hold for higher order ${}_{p+1}F_{p}$ with appropriate number of unit shifts in parameters. We envision further investigation of this topic in a future publication.

\section{Multi-term transformations of the non-terminating series}\label{section5}

N{\o}rlund proved the following identity \cite[relation~(5.8)]{Norlund}:
\begin{gather}\label{eq:Norlund5.8}
\frac{\Gamma\big(\a_{[1]}\big)}{\Gamma(1-a_{1})\Gamma(\b)}\, {}_{p+1}F_p\!
\left(\begin{matrix}\a\\\b\end{matrix}\right)
=\sum\limits_{k=2}^{p+1}\frac{\Gamma(a_k)\Gamma\big(\a_{[1,k]}-a_{k}\big)}{\Gamma(1-a_{1}+a_{k})\Gamma(\b-a_{k})}\,
{}_{p+1}F_p\!\left(\begin{matrix}a_{k},1-\b+a_{k}\\1-\a_{[k]}+a_{k}\end{matrix}\right)\!,
\end{gather}
where all series involved converge if $\operatorname{Re}(\nu(\a;\b))>0$. This identity was later rediscovered by Wimp in \cite[Lemma~2]{Wimp87}. We note that it also follows immediately on substituting \eqref{eq:Gppexpansion1} into~\eqref{eq:Frepr} and integrating term-wise using the beta integral:
\begin{gather*}%\label{eq:Frepr1}
{}_{p+1}F_p\left.\! \left(\begin{matrix}a_1,\a\\\b\end{matrix}\right.\right)
=\frac{\Gamma(\b)}{\Gamma\big(\a_{[1]}\big)}\sum\limits_{k=2}^{p+1}
\frac{\Gamma\big(a_{[1,k]}-a_k\big)}{\Gamma(\b-a_k)}\sum\limits_{n=0}^{\infty}
\frac{(1-\b+a_k)_n}{\big(1-\a_{[1,k]}+a_k\big)_n n!}\int_0^{1}\frac{t^{n+a_k-1}}{(1-t)^{a_1}}{\rm d}t
\\ \hphantom{{}_{p+1}F_p\left. \!\left(\begin{matrix}a_1,\a\\\b\end{matrix}\right.\right)}
{}=\frac{\Gamma(\b)\Gamma(1-a_1)}{\Gamma\big(\a_{[1]}\big)}\sum\limits_{k=2}^{p+1}
\frac{\Gamma(a_k)\Gamma\big(\a_{[1,k]}-a_{k}\big)}{\Gamma(1-a_{1}+a_{k})\Gamma(\b-a_{k})}\,
{}_{p+1}F_p\!\left(\begin{matrix}a_{k},1-\b+a_{k}\\1-\a_{[k]}+a_{k}\end{matrix}\right)\!.
\end{gather*}
On the other hand, formulas \eqref{eq:Gppexpansion1} and \eqref{eq:Norlund} imply immediately that for $\sum_{j=1}^{p+1}(c_j-b_j)>1$ we have
\begin{equation}\label{eq:Gppexpansion4}
\sum\limits_{k=1}^{p+1}\frac{\Gamma\big(\b_{[k]}-b_{k}\big)}{\Gamma(\cb-b_{k})}\, {}_{p+1}F_{p}\!
\left( \begin{matrix}1-\cb+b_k\\1-\b_{[k]}+b_k \end{matrix}\right)=0.
\end{equation}
We note that particular cases of identities \eqref{eq:Norlund5.8} and \eqref{eq:Gppexpansion4} appeared many times in the literature proved by various methods. For instance, the three-term Thomae relation for ${}_3F_{2}$ \cite[formula~3.2(2)]{Bailey} is a particular case of \eqref{eq:Norlund5.8}, while its variation \cite[formula~(10)]{Darling} and ${}_4F_{3}$ generalization \cite[formula~(19)]{Darling} reduce to~\eqref{eq:Gppexpansion4} after appropriate change of notation.

We note also that \eqref{eq:Gppexpansion4} is not a straightforward rewriting of \eqref{eq:Norlund5.8}. Indeed, setting $b_1=0$ brings~\eqref{eq:Gppexpansion4} to the form:
\begin{gather*}
\frac{\Gamma\big(\b_{[1]}\big)}{\Gamma(\cb)}\,{}_{p+1}F_{p}\!
\left( \begin{matrix}1-\cb\\1-\b_{[1]} \end{matrix}\right)
=\sum\limits_{k=2}^{p+1}\frac{\Gamma(1-b_{k})\Gamma\big(\b_{[1,k]}-b_{k}\big)}{b_k\Gamma(\cb-b_k)}\, {}_{p+1}F_{p}\!\left( \begin{matrix}1-\cb+b_k\\1+b_k,1-\b_{[1,k]}+b_k \end{matrix}\right)\!.
\end{gather*}
Changing notation according to $1-\cb\to\a,$ $1-\b_{[1]}\to\b$ we get
\begin{gather}\label{eq:Gppexpansion8}
\frac{\Gamma(1-\b)}{\Gamma(1-\a)}{}_{p+1}F_{p}\!\left( \begin{matrix}\a\\\b\end{matrix} \right)=-\sum\limits_{k=1}^{p}\frac{\Gamma(b_k-1)\Gamma\big(b_k-\b_{[k]}\big)}{\Gamma(b_k-\a)}\, {}_{p+1}F_{p}\!\left( \begin{matrix}\a+1-b_k\\2-b_k,1+\b_{[k]}-b_k\end{matrix} \right)\!,
\end{gather}
where all series involved converge if $\operatorname{Re}(\nu(\a;\b))>0$. This is manifestly different from~\eqref{eq:Norlund5.8}. It remains unclear for us, however, if~\eqref{eq:Norlund5.8} and~\eqref{eq:Gppexpansion8} can be obtained from each other by the appropriate compositions.

Our next theorem gives another multi-term identity for ${}_{p+1}F_{p}$ containing $p$ or more terms and not immediately found in the literature (other than for $p=2$, see remark below).
\begin{Theorem}\label{thm:Gsum1}
	Suppose $0\le{n}\le{p}$, $0\le{m}\le{p}$ are integers that satisfy $m+n\ge{p}$ and $\a\in\C^{n}$, $\b\in\C^{m}$, $\cb\in\C^{p-n}$, $\d\in\C^{p-m}$ are complex vectors that satisfy
\begin{gather*} \operatorname{Re}\Biggl(\sum_{j=1}^{n}a_j+\sum_{j=1}^{p-n}c_j-\sum_{j=1}^{m}b_j-\sum_{j=1}^{p-m}d_j\Biggr)>0.
\end{gather*}
Then
\begin{subequations}
\begin{gather}
\sum\limits_{k=1}^{m}\frac{A_{k}}{b_{k}}\, {}_{p+1}F_{p}\!\left( \left.
\begin{matrix}1-\a+b_{k},1-\cb+b_{k},b_{k}\\1-\b_{[k]}+b_{k},1-\d+b_{k},b_{k}+1\end{matrix} \right|
(-1)^{p-m-n}\!\right)\nonumber
\\ \qquad
{}+\sum\limits_{k=1}^{n}\frac{B_{k}}{(1-a_{k})}\, {}_{p+1}F_{p}\!\left( \left. \begin{matrix}1+\b-a_{k},1+\d-a_{k},1-a_{k}\\1+\a_{[k]}-a_{k},1+\cb-a_{k},2-a_{k}\end{matrix} \right|
(-1)^{p-m-n}\!\right)\nonumber
\\ \qquad
{}=\frac{\Gamma(1-\a)\Gamma(\b)}{\Gamma(\cb)\Gamma(1-\d)},
\label{eq:Newsummation}
\end{gather}
where
\begin{gather}\label{eq:constants}
A_{k}=\frac{\Gamma\big(\b_{[k]}-b_{k}\big)\Gamma(1-\a+b_k)}{\Gamma(\cb-b_{k})\Gamma(1-\d+b_{k})},\qquad
B_k=\frac{\Gamma(a_k-\a_{[k]})\Gamma(1+\b-a_{k})}{\Gamma(a_{k}-\d)\Gamma(1+\cb-a_{k})}.
\end{gather}
\end{subequations}
\end{Theorem}

\begin{proof}
Combining \cite[formula~(8.2.2.3)]{PBM3} with \cite[fornula~(8.2.2.4)]{PBM3} for $x>0$ we get:
\begin{gather}
G^{m,n}_{p,p}\!\left( x\left|\begin{matrix} \a,\cb \\ \b,\d \end{matrix}\right. \right)
=H(1-x)\sum\limits_{k=1}^{m}A_kx^{b_k}\,{}_{p}F_{p-1}\!
\left( \begin{matrix}1-\a+b_k,1-\cb+b_k\\1-\b_{[k]}+b_k,1-\d+b_k\end{matrix}
\!\vline\,(-1)^{p-m-n}x\!\right)\nonumber
\\ \hphantom{G^{m,n}_{p,p}\!\left( x\left|\begin{matrix} \a,\cb \\ \b,\d \end{matrix}\right. \right)
}
{}+H(x\!-1)\sum\limits_{k=1}^{n}B_kx^{a_k-1}\,{}_{p}F_{p-1}\!
\left( \begin{matrix}1+\b-a_k,1+\d-a_k\\1 +\a_{[k]} -a_k,1 +\cb -a_k\end{matrix} \vline\,
\frac{(-1)^{p-m-n}}{x}\!\right)\!,\!\!\!
\label{eq:Gppexpansion}
\end{gather}
where $A_k$, $B_k$ are defined in \eqref{eq:constants} and $H(t)=\partial_t\max\{t,0\}$ is the Heaviside function.
According to \cite[Theorem~2.2]{KilSaig} the Mellin transform
\begin{gather*}
\int_{0}^{\infty}x^{s-1}G^{m,n}_{p,p}\left(x\left|\begin{matrix} \a,\cb \\ \b,\d
\end{matrix}\right.\right){\rm d}x
=\frac{\Gamma(1-\a-s)\Gamma(\b+s)}{\Gamma(\cb+s)\Gamma(1-\d-s)}
\end{gather*}
exists for $-\min(\operatorname{Re}(\b))<\operatorname{Re}(s)<\min(1-\operatorname{Re}(\a))$ under conditions of the theorem. Assuming that $s=0$ belongs to this range we can take the Mellin transform with $s=0$ on both sides of~\eqref{eq:Gppexpansion} to obtain
\begin{align*}%\label{eq:GppexpIntegral}
\frac{\Gamma(1-\a)\Gamma(\b)}{\Gamma(\cb)\Gamma(1-\d)}
={}&\sum\limits_{k=1}^{m}A_k\sum_{j=0}^{\infty}\frac{(1-\a+b_k)_{j}(1-\cb+b_k)_{j}[(-1)^{p-m-n}]^{j}}
{(1-\b_{[k]}+b_k)_{j}(1-\d+b_k)_{j}j!}\int_{0}^{1}x^{b_k+j-1}{\rm d}x
\\
&+\sum\limits_{k=1}^{n}B_k\sum_{j=0}^{\infty}\frac{(1+\b-a_k)_{j}(1+\d-a_k)_{j}
[(-1)^{p-m-n}]^{j}}{(1+\a_{[k]}-a_k)_{j}(1+\cb-a_k)_{j}j!}\int_{1}^{\infty}x^{a_k-j-2}{\rm d}x
\\
={}&\sum\limits_{k=1}^{m}\frac{A_k}{b_k}\sum_{j=0}^{\infty}\frac{(1-\a+b_k)_{j}(1-\cb+b_k)_{j}[(-1)^{p-m-n}]^{j}(b_k)_{j}}{(1-\b_{[k]}+b_k)_{j}(1-\d+b_k)_{j}j!(b_k+1)_{j}}
\\
&+\sum\limits_{k=1}^{n}\frac{B_k}{1-a_k}\sum_{j=0}^{\infty}\frac{(1+\b-a_k)_{j}(1+\d-a_k)_{j}[(-1)^{p-m-n}]^{j}(1-a_k)_{j}}{(1+\a_{[k]}-a_k)_{j}(1+\cb-a_k)_{j}(2-a_k)_{j}j!},
\end{align*}
which proves \eqref{eq:Newsummation} under the restriction $-\min(\operatorname{Re}(\b))<0<\min(1-\operatorname{Re}(\a))$. This restriction can now be removed by analytic continuation.
\end{proof}

\begin{Remark} An essential part of the above calculation was made by Slater in \cite[formulas~(4.8.1.14)--(4.8.1.23)]{Slater}. For unknown reasons she decided not to make final step leading to~\eqref{eq:Newsummation}. The case $p=2$ of formula \eqref{eq:Newsummation} is known, see \cite[formula~(7.4.4.11)]{PBM3}.
\end{Remark}

\section{Multi-term transformations of the terminating series}\label{section6}

Introduce the following notation: $\m=(m_1,\dots,m_r)\in\Z^r$ and $\n=(n_1,\dots,n_r)\in\Z^r$,
\begin{gather*}
M=m_1+\cdots+{m_r},\qquad N=n_1+\cdots+{n_r},
\\
m_{\min}=\min\limits_{1\le{i}\le{r}}(m_i),\qquad
n_{\max}=\max\limits_{1\le{i}\le{r}}(n_i),\qquad
p=\max\{-1,M-N-r+1\}.
\end{gather*}
In \cite[Theorem~1]{KarpKuzn} A.\:Kuznetsov and the second author established the following identity for arbitrary $\a,\b\in\C^{r}$ such that the components of $\a$ are distinct modulo integers, and $|z|<1$:
\begin{gather}
\sum\limits_{i=1}^{r}\frac{(1-\b+a_i)_{\m-n_i}z^{-n_i}}{\big(a_i-\a_{[i]}\big)_{\n_{[i]}-n_i+1}}\,
{}_{r}F_{r-1}\!\left(\begin{matrix}\b-a_i\\1+\a_{[i]}-a_i\end{matrix}\Big\vert z\right)
{}_{r}F_{r-1}\!\left(\begin{matrix}1-\b+a_i+\m-n_i\\1-\a_{[i]}+a_i+\n_{[i]}-n_i\end{matrix}\Big\vert z\right)\nonumber
\\ \qquad
{}=(1-z)^{-p-1}\sum\limits_{j=-n_{\max}}^{p-m_{\min}}\beta_jz^{j}.
\label{eq:duality}
\end{gather}
The authors did not give any explicit expression for the numbers $\beta_j$. Our aim here is to emp\-loy the above formula for deriving some identities for the generalized hypergeometric function evaluated at $1$. To this end we will need the explicit expressions for $\beta_j$, which we present in the following lemma.

\begin{Lemma}\label{lm:beta-final}
Formula \eqref{eq:duality} holds true for
\begin{gather}
\beta_{k}=
\sum\limits_{j=\max(-n_{\max},k-p-1)}^{k}\binom{p+1}{k-j}(-1)^{k-j}\label{eq:beta-final}
\\ \hphantom{\beta_{k}=}
{}\times\sum\limits_{i=1}^{r}\frac{(1-\b+a_i)_{\m+j}}{\big(a_i-\a_{[i]}\big)_{\n_{[i]}+j+1}(j+n_i)!}\, {}_{2r}F_{2r-1}
\!\left( \begin{matrix}-j\!-n_i,\b -a_i,1 \a_{[i]} -a_{i} -\n_{[i]} -j-1
\\ \b-a_{i}-\m-j,1+\a_{[i]}-a_{i}\end{matrix} \right)\!,
\nonumber
\end{gather}
where each term with $j+n_i<0$ is assumed to equal zero.
\end{Lemma}
\begin{proof}
Writing $S(z)$ for the left hand side of \eqref{eq:duality} we get by collecting terms
\begin{gather*}
S(z)=\sum\limits_{k=-n_{\max}}^{\infty}z^k
\underbrace{\sum\limits_{i=1}^{r}\sum\limits_{j=0}^{k+n_i}\gamma^{k+n_i}_{i,j}}_{=\alpha_k}=\sum\limits_{k=-n_{\max}}^{\infty}\alpha_{k}z^k,
\end{gather*}
where, according to the formula on page 4 of \cite{KarpKuzn}, we have the first equality below:
\begin{align*}
\gamma^{k+n_i}_{i,j}&=\frac{(-1)^{j}(1-\b+a_i-j)_{\m+k}}{\big(a_i-\a_{[i]}-j\big)_{\n_{[i]}+k+1}j!(k+n_i-j)!}
\\
&=\frac{(1-\b+a_i)_{\m+k}(\b-a_i)_{j}\big(1+\a_{[i]}-a_i-\n_{[i]}-k-1\big)_{j}(-k-n_i)_{j}}
{(\b-a_i-\m-k)_{j}\big(a_i-\a_{[i]}\big)_{\n_{[i]}+k+1}\big(1+\a_{[i]}-a_i\big)_{j}(k+n_i)!j!}.
\end{align*}
The second equality above is obtained by an application of the easily verifiable identities
\begin{gather*}
(z-j)_{n}=\frac{(z)_n(1-z)_{j}}{(1-z-n)_{j}}\qquad \text{and}\qquad (m-j)!=(-1)^j\frac{m!}{(-m)_{j}}.
\end{gather*}
Hence,
\begin{gather}\label{eq:alpha-defined}
\alpha_k=\!
\sum\limits_{i=1}^{r}\frac{(1-\b+a_i)_{\m+k}}{\big(a_i\!-\a_{[i]}\big)_{\n_{[i]}+k+1}(k\!+n_i)!}\,
{}_{2r}F_{2r-1}\!\left( \begin{matrix}-k\!-n_i,\b\!-a_i,1\!+\a_{[i]}\!-a_{i}\!-\n_{[i]}\!-\!k\!-\!1
\\\b-a_{i}-\m-k,1+\a_{[i]}-a_{i}\end{matrix} \right)\!.\!\!\!\!
\end{gather}
Multiplying both sides of \eqref{eq:duality} by $(1-z)^{p+1}$ and expanding by the binomial theorem we obtain
\begin{gather*}
\sum_{j=0}^{p+1}\binom{p+1}{j}(-z)^j\sum\limits_{k=-n_{\max}}^{\infty}\alpha_{k}z^k=\sum\limits_{k=-n_{\max}}^{p-m_{\min}}\beta_{k}z^{k}.
\end{gather*}
Multiplying both sides by $z^{n_{\max}}$, changing $k+n_{\max}\to{k}$ and writing $\hat{\beta}_k=\beta_{k-n_{\max}}$,
$\hat{\alpha}_k=\alpha_{k-n_{\max}}$, we get
\begin{gather*}
\sum_{j=0}^{p+1}\binom{p+1}{j}(-1)^jz^{j}\sum\limits_{k=0}^{\infty}\hat{\alpha}_{k}z^k
=\sum\limits_{s=0}^{\infty}z^{s}\sum\limits_{j+k=s}\binom{p+1}{j}(-1)^{j}\hat{\alpha}_{k}
=\sum\limits_{k=0}^{p-m_{\min}+n_{\max}}\hat{\beta}_{k}z^{k}.
\end{gather*}
In view of $\binom{p+1}{j}=0$ for $j>p+1$, this implies that
\begin{gather*}
\beta_{s-n_{\max}}=\hat{\beta}_s=\sum\limits_{j=0}^{\min(s,p+1)}\binom{p+1}{j}(-1)^{j}\hat{\alpha}_{s-j}
=\sum\limits_{j=0}^{\min(s,p+1)}\binom{p+1}{j}(-1)^{j}\alpha_{s-j-n_{\max}}
\end{gather*}
for $s=0,\dots,p-m_{\min}+n_{\max}$. Returning to $k=s-n_{\max}$ we obtain by changing the index of summation according to the rule $j\to{k-j}$:
\begin{gather*}
\beta_{k}=\sum\limits_{j=0}^{\min(k+n_{\max},p+1)}\binom{p+1}{j}(-1)^{j}\alpha_{k-j}=
\sum\limits_{j=\max(-n_{\max},k-p-1)}^{k}\binom{p+1}{k-j}(-1)^{k-j}\alpha_{j}
\end{gather*}
for $k=-n_{\max},\dots,p-m_{\min}$. Substituting the formula for $\alpha_{j}$ we finally arrive at \eqref{eq:beta-final}.
\end{proof}

\looseness=-1 Next, we derive identities for the numbers $\alpha_k$ defined in \eqref{eq:alpha-defined} and $\beta_k$ from \eqref{eq:beta-final}. To formulate our proposition for $\alpha_k$, we will need the standard Bernoulli polynomials $\Be_n(x)=\Be_{n}^{(1)}(x)$, where~$\Be_{n}^{(\sigma)}(x)$ is defined in \eqref{eq:Bern-Norl-defined}. Note that the explicit expression \cite[formula~(24.2.3)]{NIST} implies that the leading coefficient of $\Be_n(x)$ is $1$. Further, define the polynomial $q_p(k)$ by the recurrence
\begin{subequations}\label{eq:qp-defined}
\begin{gather}
q_0=1,\qquad
q_p(k)=\frac{1}{p}\sum_{j=1}^{p}\frac{(-1)^{j+1}}{j+1}Q_{j}(k)q_{p-j}(k),
\end{gather}
where
\begin{gather}
Q_{j}(k)=\sum\limits_{i=1}^{r}\big[\Be_{j+1}(-a_i-k)-\Be_{j+1}(-b_i+1-k)+\Be_{j+1}(1-b_i+m_i)\nonumber
\\ \hphantom{Q_{j}(k)=\sum\limits_{i=1}^{r}\big[}
{}-\Be_{j+1}(1-a_i+n_i)\big]\!.
\end{gather}
\end{subequations}
In other words, $q_p(k)$ is the coefficient at $z^{-p}$ in the exponential expansion
\begin{gather*}
\exp\left[\sum_{j=1}^{\infty}\frac{(-1)^{j+1}Q_{j}(k)}{(j)_2z^j}\right]=1+\sum_{s=1}^{\infty}\frac{q_{s}(k)}{z^s},
\end{gather*}
where $p=\max\{-1,M-N-r+1\}$. Clearly, $q_{-1}=0$ and $q_0=1$. There are several other ways to compute $q_p(k)$: via the complete (or exponential) Bell polynomials, by the determinantal and the explicit multiple sum formulas, see Section~\ref{section2} and \cite[Section~3.3]{Comtet}. We have the following proposition.

\begin{Theorem}\label{th:alpha}
For all $k\ge-m_{\min}$ we have $\alpha_k=q_p(k)$ , or, in detail,
\begin{gather*}
\sum\limits_{i=1}^{r}\frac{(1-\b+a_i)_{\m+k}}{\big(a_i-\a_{[i]}\big)_{\n_{[i]}+k+1}(k+n_i)!}\, {}_{2r}F_{2r-1}
\!\left( \begin{matrix}-k\!-n_i,\b\!-a_i,1\!+\a_{[i]}\!-a_{i}\!-\n_{[i]}\!-k\!-1
\\\b-a_{i}-\m-k,1+\a_{[i]}-a_{i}\end{matrix} \right) =q_p(k),
\end{gather*}
where $q_p$ is defined in \eqref{eq:qp-defined}. In particular, $q_{-1}(k)=0$, $q_{0}(k)=1$, and, furthermore,
\begin{gather*}
q_1(k)=\frac{Q_1(k)}2,\qquad q_2(k)=\frac{Q_{1}^2(k)}8-\frac{Q_2(k)}6.
\end{gather*}
\end{Theorem}

\begin{proof}
Keeping the notation from the proof of Lemma~\ref{lm:beta-final} we can write
\begin{gather*}
S(z)=\sum\limits_{k=-n_{\max}}^{-m_{\min}-1}\alpha_{k}z^k+\sum\limits_{k=-m_{\min}}^{\infty}\alpha_{k}z^k,
\end{gather*}
where the first sum is zero if $-n_{\max}>-m_{\min}-1$. Furthermore
\cite[Lemma~1]{KarpKuzn} shows that $\alpha_{k}$ in the second sum is precisely the polynomial $q_p(k)$ defined above, see the first formula below the proof of \cite[Lemma~1]{KarpKuzn}.
\end{proof}
Writing as before $\nu=\nu(\a;\b)=\sum_{i=1}^{r}(b_i-a_i)$, we have

\begin{Theorem}\label{th:sumbeta}
Let $\a,\b\in\C^r$, $\m,\n\in\Z^r$ and
\begin{gather}\label{eq:bothdivergent}
\nu-M+N-1<0, \qquad -\nu+r-1<0.
\end{gather}
Then the numbers $\beta_k$ defined in \eqref{eq:beta-final} satisfy the following identity
\begin{gather}\label{eq:betasum}
\sum\limits_{k=-n_{max}}^{p-m_{min}}\beta_k=(-1)^{r-1}(\nu)_{1-r}(1-\nu)_{M-N}.
\end{gather}
\end{Theorem}

\begin{proof} According to \eqref{eq:f-g} the asymptotic relation
\begin{gather*}
{}_{r}F_{r-1}\left(\left.\begin{matrix}\cb\\\d\end{matrix}\right\vert z\right)=\frac{\Gamma(\d)\Gamma(-\nu(\cb;\d))}{\Gamma(\cb)}(1-z)^{\nu(\cb;\d)}(1+o(1))~~\text{as}~~z\to1,
\end{gather*}
is valid if $\nu(\cb;\d)=\sum_{j=1}^{r-1}d_j-\sum_{j=1}^{r}c_j<0$. Hence, under conditions \eqref{eq:bothdivergent} the above relation is applicable to both series in each summand in \eqref{eq:duality}. Summing both inequalities in~\eqref{eq:bothdivergent} we see that
$M-N-r+1>-1$ which implies by definition of $p$ that $p=M-N-r+1$. A~straightforward calculation using the asymptotic relation above then shows that the total power of $1-z$ does not depend on $i$ and equals $-p-1$. Hence, we get from \eqref{eq:duality} as $z\to1$
\begin{multline*}
\Gamma(-\nu+M-N+1)\Gamma(\nu-r+1)(1+o(1-z))
\\
\times\sum\limits_{i=1}^{r}\frac{(1-\b+a_i)_{\m-n_i}z^{-n_i}}{\big(a_i-\a_{[i]}\big)_{\n_{[i]}-n_i+1}}
\frac{\Gamma\big(1+\a_{[i]}-a_i\big)\Gamma\big(1-\a_{[i]}+a_i+\n_{[i]}-n_i\big)}{\Gamma(\b-a_i)\Gamma(1-\b+a_i+\m-n_i)}=\sum\limits_{j=-n_{\max}}^{p-m_{\min}}\beta_jz^{j}.
\end{multline*}
In the limit $z=1$ we thus obtain after simple rearrangement and application of the reflection formula
$\Gamma(x)\Gamma(1-x)=\pi/\sin(\pi{x})$ that
\begin{gather*}
(-1)^{r-1}\Gamma(-\nu+M-N+1)\Gamma(\nu-r+1)
\sum\limits_{i=1}^{r}\frac{\sin[\pi((\b-a_i))]}{\pi\sin[\pi((\a_{[i]}-a_i))]}=\sum\limits_{j=-n_{\max}}^{p-m_{\min}}\beta_j.
\end{gather*}
Finally, an application of the identity \cite[identity~(3.14)]{KPSIGMA}
\begin{gather*}
\sum\limits_{i=1}^{r}\frac{\sin[\pi((\b-a_i))]}{\sin[\pi((\a_{[i]}-a_i))]}=\sin(\pi\nu)
\end{gather*}
yields \eqref{eq:betasum}.
\end{proof}

In a rather recent work \cite{Guo2015} the authors discovered some curious three-term duality relations for the $q$-hypergeometric functions. As $q\to1$ they naturally lead to identities for ordinary generalized hypergeometric functions which do not seem to appear in the literature previously. The general case and its corollary are presented below. We will then use them to derive three-term relations for terminating hypergeometric series evaluated at $\pm1$.

\begin{Lemma}%\label{lm:Guo}
	Let $r,s\ge 0$ be integers, $a,b,c,d\in\C$, $\e\in\C^r$, $\f\in\C^s$. Then we have in the sense of formal power series\emph{:}
\begin{gather}
a(b-d)(c-d)(b+c-a)\, {}_{r+4}F_{s+3}\left.\! \left(\begin{matrix} b+c-a-1,b+c-2,c,d,\e\\ a,b-1,b+c-d-1,\f\end{matrix}\right\vert z\right)\nonumber
\\ \qquad
{}\times{}_{r+4}F_{s+3}\left.\! \left(\begin{matrix} b+c-a+1,b+c,c,d,\e+1\\ a,b+1,b+c-d+1,\f+1\end{matrix}\right\vert z\right)\nonumber
\\ \qquad
{}-d(b-a)(c-a)(b+c-d){}_{r+4}F_{s+3}\left.\! \left(\begin{matrix} b+c-a,b+c-2,c,d-1,\e\\ a-1,b-1,b+c-d,\f\end{matrix}\right\vert z\right)\nonumber
\\ \qquad\hphantom{=}
{}\times{}_{r+4}F_{s+3}\left.\! \left(\begin{matrix} b+c-a,b+c,c,d+1,\e+1\\ a+1,b+1,b+c-d,\f+1\end{matrix}\right\vert z\right)\nonumber
\\ \qquad
{}=bc(a-d)(b+c-a-d)\, {}_{r+4}F_{s+3}\left.\! \left(\begin{matrix} b+c-a,b+c-2,c-1,d,\e\\ a-1,b,b+c-d-1,\f\end{matrix}\right\vert z\right)\nonumber
\\ \qquad\hphantom{=}
{}\times{}_{r+4}F_{s+3}\left.\! \left(\begin{matrix} b+c-a,b+c,c+1,d,\e+1\\ a+1,b,b+c-d+1,\f+1\end{matrix}\right\vert z\right)\!.
\label{eq:dualityGuo}\end{gather}
\end{Lemma}
\begin{proof}
Replacing $a$, $b$, $c$, $d$ by $q^{a}$, $q^{b}$, $q^{c}$, $q^{d}$; $e_1,e_2,\dots,e_r$ by $q^{e_1}, q^{e_2},\dots,q^{e_r}$ and $f_1,f_2,\dots,f_s$ by $q^{f_1},q^{f_2},\dots,q^{f_s}$ and letting $q\to1$ in \cite[Theorem 1.1]{Guo2015} we get the result.
\end{proof}

By specializing parameters in \cite[Theorem 1.1]{Guo2015} the authors get \cite[Corollary 1.2]{Guo2015}.
Its~limi\-ting case is the following:
\begin{Lemma}\label{lm:Guo-reduced}
For a given integer $r\ge1$ suppose $\a\in\C^{r+1}$ and $\b\in\C^{r}$. Then
\begin{gather*}
a_1(a_2-b_1)\,{}_{r+1}F_r\left.\!\!\left(\begin{matrix} a_1\!-1,\a_{[1]}\\ b_1\!-1,\b_{[1]}\end{matrix}\right\vert z\right) {}_{r+1}F_r\left.\!\!\left(\begin{matrix} a_2,\a_{[2]}\!+1\\ \b+1\end{matrix}\right\vert z\right)\!
-a_2(a_1-b_1)\,{}_{r+1}F_r\left.\!\!\left(\begin{matrix} a_2-1,\a_{[2]}\\ b_1-1,\b_{[1]}\end{matrix}\right\vert z\right)\!
\\ \qquad
{}\times {}_{r+1}F_r\left.\!\!\left(\begin{matrix} a_1,\a_{[1]}+1\\ \b+1\end{matrix}\right\vert z\right)\nonumber
=b_1(a_2-a_1)\,{}_{r+1}F_r\left.\!\!\left(\begin{matrix} \a\\ \b\end{matrix}\right\vert z\right)\,{}_{r+1}F_r\left.\!\!\left(\begin{matrix} a_1,a_2,\a_{[1,2]}+1\\ b_1,\b_{[1]}+1\end{matrix}\right\vert z\right)\!.
\end{gather*}
\end{Lemma}
\begin{proof} Replacing $a_1,a_2,\dots,a_r$ by $q^{a_1},q^{a_2},\dots,q^{a_r}$ and $b_1,b_2,\dots,b_r$ by $q^{b_1},q^{b_2},\dots,q^{b_r}$, and letting $q\to1$ in \cite[Corollary~1.2]{Guo2015} we arrive at the claim.
\end{proof}
\begin{Theorem}\label{th:Guo-general}
Let $r,s\ge0$ be integers, $a,b,c,d\in\C$, $\e\in\C^r$, $\f\in\C^s$. Define $t=r+s+8$, $\hat{a}_n=1-a-n$, $\hat{b}_n=1-b-n$, $\hat{c}_n=1-c-n$, $\hat{d}_n=1-d-n$ for any integer $n\ge0$. Then we have
\begin{gather*}%\label{eq:theoremfor}
\frac{(b-d)(c-d)(a-1)_{2}(b+c-a-1)_{2}}{(a+n-1)(b+c-a+n-1)}
\\ \qquad
{}\times{}_{t}F_{t-1}\left.\!\!\left(\begin{matrix} -n, \hat{b}_n-c+d+1, b+c-a+1, c, d, \hat{a}_n,\hat{b}_n+1, b+c, \e, 2-\f-n \\b+c-d+1, \hat{b}_n-c+a+1, \hat{c}_n,\hat{d}_n, a, b+1,\hat{b}_n-c+2, 2-\e-n, \f \end{matrix}\right\vert (-1)^{t} \right)
\\ \qquad
{}-\frac{(b-a)(c-a)(d-1)_{2}(b+c-d-1)_{2}}{(d+n-1)(b+c-d+n-1)}
\\ \qquad
{}\times{}_{t}F_{t-1}\left.\!\!\left(\begin{matrix} -n, \hat{b}_n-c+d, b+c-a, c, d+1, \hat{a}_n+1, \hat{b}_n+1, b+c, \e, 2-\f-n\\ b+c-d, \hat{b}_n-c+a, \hat{c}_n, \hat{d}_n+1, a+1, b+1, \hat{b}_n-c+2, 2-\e-n, \f \end{matrix}\right\vert (-1)^{t} \right)
\\ \qquad
{}=\frac{(a-d)(b+c-a-d)(b-1)_{2}(c-1)_{2}}{(b+n-1)(c+n-1)}		
\\ \qquad
{}\times{}_{t}F_{t-1}\left.\!\!\left(\begin{matrix} -n, \hat{b}_n-c+d+1, b+c-a, c+1,d,\hat{a}_n+1,\hat{b}_n, b+c, \e,2-\f-n\\ b+c-d+1, \hat{b}_n-c+a, \hat{c}_n+1,\hat{d}_n,a+1, b, \hat{b}_n-c+2, 2-\e-n, \f \end{matrix}\right\vert (-1)^{t} \right)\!.
\end{gather*}
\end{Theorem}
\begin{proof}
	Equating coefficients at $z^n$ on both sides of \eqref{eq:dualityGuo}, using the standard Cauchy product and relations \eqref{eq:pochammer}
	we arrive at the result.
\end{proof}

In a similar fashion from Lemma~\ref{lm:Guo-reduced} we obtain the following:
\begin{Theorem}\label{th:Guo-particular}
For a given integer $r\ge 1$ suppose $\a\in\C^{r+1}$, $\b\in\C^{r}$. Then for all integer $n\ge0$
\begin{gather*}
(a_2-b_1)\frac{a_1(a_1-1)}{(a_1+n-1)}\, {}_{2r+2}F_{2r+1}\!\!\left(\begin{matrix} -n,2-b_1-n,1-\b_{[1]}-n,a_2,\a_{[2]}+1\\2-a_1-n,1-\a_{[1]}-n,\b+1\end{matrix} \right)
\\ \qquad
{}-(a_1-b_1)\frac{a_2(a_2-1)}{(a_2+n-1)}\,
{}_{2r+2}F_{2r+1}\!\!\left(\begin{matrix} -n,2-b_1-n,1-\b_{[1]}-n,a_1,\a_{[1]}+1\\2-a_2-n,1-\a_{[2]}-n,\b+1\end{matrix} \right)
\\ \qquad
{}=(a_2-a_1)\frac{b_1(b_1-1)}{(b_1+n-1)}\,
{}_{2r+2}F_{2r+1}\!\!\left(\begin{matrix} -n,1-\b-n,a_1,a_2,\a_{[1,2]}+1\\
1-\a-n,b_1,\b_{[1]}+1\end{matrix}\right)\!.
\end{gather*}
\end{Theorem}

\subsection*{Acknowledgements}
The second and the third named authors have been supported by the Ministry of Science and Higher Education of the Russian Federation (agreement No.~075-02-2021-1395). The third named author has been also supported by RFBR (project 20-01-00018).

\pdfbookmark[1]{References}{ref}
\LastPageEnding

\end{document}